\documentclass[11pt]{amsart}
\usepackage{amsmath, amsthm, amssymb}
\usepackage{graphicx}
\usepackage{mathbbol}
\usepackage{txfonts}
\usepackage[usenames]{color}

\newtheorem{thm}{Theorem}[section]
\newcommand{\bt}{\begin{thm}}
\newcommand{\et}{\end{thm}}

\newtheorem{cor}[thm]{Corollary}   

\newcommand{\bc}{\begin{cor}}
\newcommand{\ec}{\end{cor}}

\newtheorem{lem}[thm]{Lemma}   

\newcommand{\bl}{\begin{lem}}
\newcommand{\el}{\end{lem}}

\newtheorem{prop}[thm]{Proposition}
\newcommand{\bp}{\begin{prop}}
\newcommand{\ep}{\end{prop}}

\newtheorem{defn}[thm]{Definition}
\newtheorem{conj}[thm]{Conjecture}

\newcommand{\bd}{\begin{defn}}    
\newcommand{\ed}{\end{defn}}

\newtheorem{rmrk}[thm]{Remark}   

\newcommand{\br}{\begin{rmrk}}
\newcommand{\er}{\end{rmrk}}

\newtheorem{example}[thm]{Example}

\newcommand{\GHto}{\stackrel { \textrm{GH}}{\longrightarrow} }
\newcommand{\Fto}{\stackrel {\mathcal{F}}{\longrightarrow} }

\newcommand{\Depth}{\operatorname{Depth}}

\newcommand{\be}{\begin{equation}}
\newcommand{\ee}{\end{equation}}

\newcommand{\R}{\mathbb{R}}

\newcommand{\E}{\mathbb{E}}

\newcommand{\Z}{\mathbb{Z}}

\newcommand{\diam}{\operatorname{diam}}

\DeclareMathOperator{\set}{set}

\newcommand{\Lip}{\operatorname{Lip}}

\newcommand{\mass}{{\mathbf M}}

\newcommand{\red}[1]{{\textcolor{red}{#1}}}

\newcommand{\intcurr}{{\mathbf I}}      

\newcommand{\vol}{\operatorname{Vol}}

\newcommand{\rstr}{\:\mbox{\rule{0.1ex}{1.2ex}\rule{1.1ex}{0.1ex}}\:}

\newcommand{\spt}{\operatorname{spt}}

\def\rr{\mathbb{R}}

\def\implies{\Longrightarrow}
\def\isom{\cong}

\def\too{\longrightarrow}

\def\Gr{\mathcal{G}_n(r_0,\gamma,D, \alpha)}

\begin{document}

\title[Positive mass stability for graphs]{
Intrinsic flat stability of the positive mass theorem 
for graphical hypersurfaces of Euclidean space}

\author[Huang]{Lan-Hsuan Huang}
\thanks{Huang is partially supported by NSF DMS 1308837 and DMS 1452477.}
\address{University of Connecticut}
\email{lan-hsuan.huang@uconn.edu}

\author[Lee]{Dan A. Lee}
\thanks{Lee is partially supported by a PSC CUNY Research Grant.}
\address{CUNY Graduate Center and Queens College}
\email{dan.lee@qc.cuny.edu}

\author[Sormani]{Christina Sormani}
\thanks{Sormani is partially supported by a PSC CUNY Grant and NSF DMS 1309360}
\address{CUNY Graduate Center and Lehman College}
\email{sormanic@member.ams.org}

\thanks{This material is also based upon work supported by the NSF under Grant No.~0932078~000, while all three authors were in residence at the Mathematical Sciences Research Institute in Berkeley, California, during the Fall 2013 program in Mathematical General Relativity.}
\date{}

\keywords{}

\vspace{.2cm}

\begin{abstract}
The rigidity of the Positive Mass Theorem states that the only complete asymptotically flat manifold of nonnegative scalar curvature and \emph{zero} mass is Euclidean space. We study the stability of this statement for spaces that can be realized as graphical hypersurfaces in $\E^{n+1}$.    
We prove (under certain technical hypotheses) that if a sequence of complete asymptotically flat graphs of nonnegative scalar curvature has mass approaching zero, then the sequence must converge to Euclidean space in the pointed intrinsic flat sense.   The appendix
includes a new Gromov-Hausdorff and intrinsic flat 
compactness theorem for sequences of metric spaces with uniform Lipschitz bounds on their metrics.
\end{abstract}

\maketitle

\noindent

\vspace{.5cm}

\section{Introduction}

The Positive Mass Theorem of Schoen-Yau and later Witten~\cite{Schoen-Yau-positive-mass, Witten-positive-mass} states that any complete asymptotically flat manifold of nonnegative scalar curvature has nonnegative ADM mass. 
Furthermore, if the ADM mass is zero, then the manifold must be Euclidean space. The second statement may be thought of as a rigidity theorem, and it is natural to consider the \emph{stability} of this rigidity statement.  That is, if the ADM mass is small, 
in what sense can we say that the manifold is ``close'' to Euclidean space? What topology is appropriate in this setting?

In \cite{LeeSormani1} the last two named authors conjectured that if a sequence of Riemannian manifolds with nonnegative
scalar curvature and no interior closed minimal surfaces
has ADM mass approaching zero then regions in these spaces converge in the \emph{intrinsic
flat} sense to Euclidean space.  The intrinsic flat distance,
$d_{\mathcal{F}}$, between oriented Riemannian manifolds
with boundary
was introduced by the last named author and S.~Wenger in \cite{SorWen2} applying work of 
Ambrosio-Kirchheim \cite{AK}.  Under intrinsic flat
convergence, thin regions of small volume disappear, so
it is well designed to study stability problems like this one
where it is possible that increasingly
thin gravity wells of increasingly small mass could persist
as the ADM mass converges to $0$.   The conjecture
as stated in \cite{LeeSormani1} implies the following conjecture.

\newpage
\begin{conj}[ \cite{LeeSormani1} ]\label{conj-main}
Let $M_j$ be asymptotically flat
$n$-dimensional Riemannian manifolds with nonnegative
scalar curvature and no interior closed minimal surfaces
and either no boundary or the boundary is an outermost
minimizing surface.  Fix an  $A_0>0$, and choose $p_j \in \Sigma_j$ 
to lie on a special surface $\Sigma_j\subset M_j$
such that $\vol_{n-1}(\Sigma_j)=A_0$.  If 
\be
\mathrm{m}_{\mathrm{ADM}}(M_j)\to 0
\ee
then $(M_j, p_j)$ converges to Euclidean space $(\E^n,0)$ 
in the pointed intrinsic flat sense.  That is,
for almost every $D>0$ we have
\be
d_{\mathcal{F}}
\left( B_{p_j}(D)\subset M_j, B_0(D)\subset\E^n \right) \to 0.
\ee
\end{conj}

The conjecture is deliberately vague as to the exact nature of
the sets $\Sigma_j$ in the conjecture.   The last two named authors proved the
conjecture in the rotationally (i.e. spherically) symmetric case.  They assume $\Sigma_j$ were rotationally symmetric level sets~\cite{LeeSormani1}.  They provided an example of a sequence of manifolds with
increasingly thin wells to demonstrate that this conjecture
is false if the points are not carefully selected to avoid falling
within wells.  This example also demonstrates that balls do not
converge in the Gromov-Hausdorff sense or smooth sense
to balls in Euclidean space.

There are various types of stability results in the literature. 
 H.~Bray and F.~Finster~\cite{Bray-Finster} used spinor methods and proved that if a complete three-dimensional asymptotically flat manifold  of non-negative scalar curvature has small mass and bounded isoperimetric constant and curvature, then the manifold must be close to Euclidean space in the sense that there is an upper bound for the $L^2$ norm of the curvature tensor over the manifold except for a set of small measure. This was generalized to higher dimensions by Finster and I.~Kath~\cite{Finster-Kath}. Finster~\cite{Finster} removed the dependence on the isoperimetric constant and obtained the $L^2$ bound of the curvature tensor with the exception of a set of small surface area. J.~Corvino~\cite{Corvino} proved that a particular bound on the mass and sectional curvature of a three-dimensional asymptotically flat manifold of nonnegative scalar curvature implies the manifold is diffeomorphic to $\mathbb{R}^3$.
Under the assumption of conformal flatness and zero scalar curvature outside a compact set,  the second author \cite{Lee-near-equality} proved that if a sequence of smooth asymptotically flat metrics of nonnegative scalar curvature has mass approaching zero, then the sequence converges in smooth topology to the Euclidean metric outside a compact set. Those results can be viewed as the stability results in the region of the manifold where the curvature tensor is uniformly bounded.

Conjecture~\ref{conj-main} addresses a different, and perhaps more challenging, aspect of the stability problem, which intends to understand how  the ADM mass controls the region of the manifold where the curvature may be large. Until now the conjecture has only been verified for rotationally symmetric spaces---an extremely restricted class. 

We now consider the much larger (but still fairly restricted) class of graphical hypersurfaces of Euclidean space. For this class of asymptotically flat manifolds of nonnegative scalar curvature, G.~Lam~\cite{Lam-graph} proved the positive mass inequality in all dimensions, and the first named author and D.~Wu~\cite{Huang-Wu:2013} proved rigidity: {\em if the ADM mass is zero, then the hypersurface must be a hyperplane}. Recently the first two named authors  proved a stability result for graphical hypersurfaces with respect to the Federer-Fleming's flat topology in $\E^{n+1}$ \cite{Huang-Lee-Graph}. However, even in Euclidean space, the flat topology and intrinsic flat topology do not have a simple relationship (see Example~\ref{example:intrinsic-flat-topology}), so the result of  \cite{Huang-Lee-Graph} does not provide a special case of the conjecture above, though it has a similar flavor. Since the flat topology is extrinsic, that result is natural from the point of view of hypersurface geometry, but it does not directly say anything about the underlying Riemannian manifolds. The purpose of this work is to prove a stability result with respect to intrinsic flat topology. We achieve this by taking the estimates used in \cite{Huang-Lee-Graph} and combining them with recent results of the last named author \cite{Sormani-AA}.

We define our class of uniformly asymptotically flat graphical hypersurfaces of $\E^{n+1}$ with
uniformly bounded depth and nonnegative
scalar curvature as follows. We first recall that the spatial $n$-dimensional Schwarzschild manifold (with boundary) of mass $m>0$ can be isometrically embedded into $\mathbb{E}^{n+1}$ as the graph of a smooth function defined on $\mathbb{E}^n \smallsetminus B((2m)^{1/(n-2)})$, with minimal boundary, such that the boundary lies in the plane $\E^n\times \left\{0\right\}$. Explicitly, it is the graph of the function $S_m(|x|)$ given in (\ref{S_m-eqn})  where
$
S_m(r)=\sqrt{ 8m (r - 2m)}
$
in dimension 3.

\begin{defn}\label{def:hypotheses}
For $n\ge3$, $r_0, \gamma, D>0$, and $\alpha<0$, define $\Gr$ to be the space of all smooth complete Riemannian manifolds of nonnegative scalar curvature, $(M^n,g)$, possibly with boundary, that admit a smooth Riemannian isometric embedding $\Psi:M\too \E^{n+1}$ such that for some open $U\subset B(r_0/2)\subset \E^n$, the image $\Psi(M)$ is the graph of a function $f\in C^\infty(\E^{n}\smallsetminus \overline{U}) \cap C^0(\E^{n}\smallsetminus {U})$:
\be
\Psi(M)=\left\{(x,f(x)): \,\, x\in \E^n \smallsetminus U\right\}
\ee
with empty or minimal boundary:
\be 
\textrm{either }
\partial M=\emptyset \textrm{ and } U=\emptyset, 
\ee
\be \label{cond1}
\textrm{ or
$f$ is constant on each component of $\partial U$ and }
\lim_{x\to\partial U } |D f(x)|=\infty,
\ee
and for almost every $h$, the level set 
\be\label{cond6}
f^{-1}(h)\subset \E^n
\textrm{ is strictly mean-convex and outward-minimizing,}
\ee
where strictly mean-convex means that the mean curvature is strictly positive, and outward-minimizing means that any region of $\E^n$ that contains the region enclosed by $f^{-1}(h)$ must have perimeter at least as large as $\mathcal{H}^{n-1}( f^{-1}(h))$.

In addition we require uniform asymptotic flatness conditions:
\be\label{cond3}
 |D f| \le \gamma
\textrm{ for }|x|\ge r_0/2  \textrm{ and }
\lim_{x\to\infty} |D f| =0
.
\ee
If $n\ge 5$, we require that 
$f(x)$ approaches a constant as $x\to\infty$. 
If $n=3$ or $4$, we require that the graph is asymptotically Schwarzschild:
\be\label{cond5}
\exists \Lambda\in \R \textrm{ such that }
 \left|f(x) - (\Lambda+ S_m(|x|)) \right| \le \gamma |x|^\alpha 
 \textrm{ for } |x|\ge r_0.
\ee

Finally we require that the regions
\be
\Omega= \Omega(r_0)=\Psi^{-1}(B(r_0)\times\rr)\,\,\,
\textrm{ and } \,\,\, \Sigma=\Sigma(r_0)=\partial\Omega(r_0)\smallsetminus \partial M
\ee
have bounded depth
 \be\label{cond7}
 \Depth(\Omega, \Sigma) =\sup\left\{d_M(p,\Sigma): p\in \Omega\right\}
 \le D.
 \ee
\end{defn}

\begin{figure}[here] 
   \centering
   \includegraphics[width=0.8\textwidth]{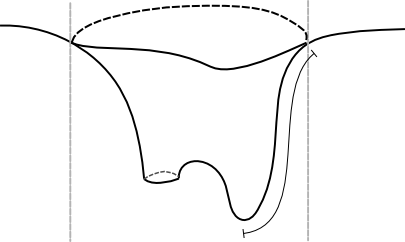}
      			\put(-40, 160){$M$}
			 \put(-210, 40){$\partial M$}
			 \put(-280, 20){$B_{r_0}\times \mathbb{R}$}
			 \put(-80, 60){$\textup{Depth}(\Omega, \Sigma)$}
			 \put(-160, 95){$\Omega(r_0)$}
			 \put(-210, 145){$\Sigma(r_0)$}
   \caption{}
   			\label{figure:graph}
\end{figure}

This class of asymptotically flat manifolds contains many nontrivial examples.  For example, we may start with an arbitrary rotationally symmetric asymptotically flat metric with nonnegative scalar curvature and no closed interior minimal surfaces.  Such a manifold embeds as a graph into Euclidean space.   Then we may perturb it as a graph slightly in any region when the scalar curvature is strictly positive.   

Our first main result is the following. 
\begin{thm}\label{thm-main}   
Let $n\ge3$, $r_0, \gamma, D>0$, $\alpha<0$, and $r\ge r_0$. For any $\epsilon>0$, there exists a $\delta=\delta(\epsilon, n, \gamma, D, \alpha, r)>0$ such that
  if $M\in\Gr$ has ADM mass 
less than $\delta$, then
\be
 d_{\mathcal{F}}\left(\,\Omega(r) \subset M\,,\, B(r)\subset \E^n\,\right)\,<\, 
 \epsilon
 \ee
 and
 \be
 |\vol(\Omega(r)) -\vol(B(r))|<\epsilon  
 \ee
 where $B(r)$ is the ball of radius $r$ around the origin, and  $\Omega(r):=\Psi^{-1}(B(r)\times\rr)$ using the notation of the above definition.
  \end{thm}

The definition of $\Gr$ essentially encodes the hypotheses of Theorem~\ref{thm-main}, so we take a moment to discuss the conditions in $\Gr$.    The condition  $U\subset B(r_0/2)$ ensures that $\partial M$ and $\Sigma(r)$, which together comprise $\partial\Omega(r)$, do not touch each other.
Conditions \eqref{cond3} and \eqref{cond5} are the asymptotic flatness conditions that we need for our proof. Note that they all follow from the fairly natural (but much stronger) requirement that the $f$'s  are uniformly asymptotically Schwarzschild up to first order.

The geometric conditions on the level sets in \eqref{cond6} are needed in order to apply the estimates of the first two authors in~\cite{Huang-Lee-Graph} and are discussed there. In particular, the first named author and Wu have proven that other conditions imply that the level sets are always weakly mean-convex~(see \cite[Theorem 4, Theorem 2.2]{Huang-Wu:2013} and \cite[Theorem 3]{Huang-Wu-Penrose}). We also use the outward minimizing property to estimate volumes in the proof of Theorem~\ref{thm-main}. 

Condition \eqref{cond7} prevents the possibility of ``arbitrarily deep gravity wells".   The notion of depth was introduced by the last named author and P.~LeFloch in~\cite{LeFloch-Sormani-1} where they proved a compactness theorem for a family of rotationally symmetric regions of nonnegative scalar curvature.  Here we use this condition
combined with the volume estimates to apply a compactness theorem of S. Wenger proven in \cite{Wenger-compactness}.

Applying Theorem~\ref{thm-main} and key
results concerning intrinsic flat convergence we obtain the
following pointed convergence theorem which proves the
conjecture for $M\in \Gr$ where $\Sigma_j$ are preimages of the intersections
of the graph $\Psi_j(M_j)$ with the cylinder at $r_0$:

\begin{thm}\label{thm-pted}   
Let $n\ge3$, $r_0, \gamma, D>0$, and $\alpha<0$. Let $M_j\in\Gr$ be a sequence such that
  \be
  m_{ADM}(M_j) \to 0.
  \ee
 If $p_j\in M_j$ is a sequence of points such that  $p_j \in \Sigma(r_0)
 :=\Psi^{-1}(\partial B(r_0)\times \mathbb{R})$, 
  then
  $(M_j,p_j)$ converges in the pointed intrinsic flat
  sense to $\E^n$.  That is,
  for almost every $R>2 r_0+D$,
  \be\label{eq-pted}
 d_{\mathcal{F}}\left(\, B_{p_j}(R) \subset M_j\,,\, B(R)\subset \E^n\,\right)\,
 \to 0
 \ee
 and
 \be
 \vol(B_{p_j}(R))\to \vol(B(R)).
 \ee
   \end{thm}

The paper begins with background material in Section~\ref{sect-back}.   We first review Federer-Fleming integral
currents and flat convergence in Euclidean space \cite{FF} and
Ambrosio-Kirchheim integral currents on metric spaces \cite{AK}.
Then we review key definitions and theorems  of the third author and
Wenger 
concerning intrinsic flat convergence
\cite{SorWen2}\cite{Wenger-compactness}\cite{Sormani-AA}
and  work of Gromov and Grove-Petersen on Gromov-Hausdorff
convergence \cite{Gromov-metric}\cite{Grove-Petersen}. 
We close with a review of prior work of the first two named
authors on graph manifolds with small ADM mass \cite{Huang-Lee-Graph}.  

In Section~\ref{sect-lim} we apply the volume and depth
bounds combined with 
Wenger's Compactness Theorem \cite{Wenger-compactness}
and an intrinsic flat Arzela-Ascoli Theorem of the third author
\cite{Sormani-AA} to prove Theorem~\ref{exists-lim}: 
{\em if $M_j \subset \Gr$ and 
$\Omega_j(r) =\Psi_j^{-1}(B(r)\times\rr)$ for fixed $r\ge r_0$
then a subsequence converges 
\be
\Omega_j(r)\Fto \Omega_\infty(r)
\ee
with a Lipschitz map $\Psi_\infty:\Omega_\infty(r)\to \E^{n+1}$}.  
Thus $\partial \Omega_j(r)\to \partial \Omega_\infty(r)$.  

In Section~\ref{sect-geom} we use the fact that the
manifolds are graphs and have ADM mass converging to $0$
applying prior work of the first two named
authors \cite{Huang-Lee-Graph}.
In Lemma~\ref{inner-gone} we prove that the inner boundaries disappear and the
outer boundaries converge:
\be
\Sigma_j(r)\Fto \Sigma_\infty(r)=\partial \Omega_\infty(r)
\ee
In Lemmas~\ref{Vol-Bounds+} and~\ref{Vol-Bounds-}
we bound the volumes of $\Omega_j(r)$ from above and below: showing $\vol(\Omega_j)\to \vol(B(r))$.   In Lemma~\ref{thm-in-disk}
we prove that $\Psi_\infty(\Omega_\infty)$ lies in a Euclidean
disk.

In Section~\ref{sect-biLip} we apply (\ref{cond3}) which
controls the gradient of the graph near the boundary of
$\Sigma_j$ to prove that the outer boundaries, $\Sigma_j(r)$
converge in the bi-Lipschitz sense to their limit $\Sigma_\infty(r)$.
This requires a new highly technical Theorem~\ref{app-thm}
concerning intrinsic flat and Gromov-Hausdorff convergence
that is proven in the appendix.  Note that one consequence
of this section is that none of the points on $\Sigma_j$
are disappearing in the intrinsic flat limit (even though points
within $\Omega_j$ may be disappearing in the limit).

In Section~\ref{sect-main} we prove Theorem~\ref{thm-main}
by combining the above results.   In Section~\ref{sect-pted}
we prove Theorem~\ref{thm-pted} by applying Theorem~\ref{thm-main}.   Note that the final steps in the proofs of these two theorems
apply far more generally than to graph manifolds as long as one
can prove the lemmas leading up to these results in the more
general case.   

\subsection{Acknowledgements}

The authors appreciate the Mathematical Sciences
Research Institute for the wonderful research environment 
there and the opportunity to begin working together on this
project.   We are grateful to Jim Isenberg,
Yvonne Choquet-Bruhat, Piotr Chrusciel, Greg Galloway, Gerhard Huisken, Sergiu Klainerman, Igor Rodnianski, and Richard Schoen
for their organization of the program in General Relativity at MSRI in Fall 2013. We also thank the referee for careful reading and for helpful comments. 


\section{Background}\label{sect-back}

Here we provide background stating the key results and notions needed from prior work that we apply in this paper.
We begin with a review on Federer-Fleming's notion of integral currents on Euclidean space
and flat convergence.  In particular we review the flat convergence of
such graphs to a plane.

Next we review intrinsic flat convergence.   We begin with
Ambrosio-Kirchheim's notion of integral currents on complete metric spaces and a review of their semicontinuity of
mass \cite{AK}.  We then review the work of the third author
with Wenger which introduced integral current spaces and
the intrinsic flat distance \cite{SorWen2} and 
key theorems about the intrinsic flat distance applied in
this paper from
\cite{SorWen2}, \cite{Wenger-compactness} and \cite{Sormani-AA}.

 Finally we present the properties of asymptotically graphs and the  key results from
work of the first two authors \cite{Huang-Lee-Graph} studying graphical hypersurfaces of $\E^{n+1}$ with nonnegative scalar curvature and small ADM mass.


\subsection{Flat Convergence of Federer and Fleming}

The notion of an integral current on $\E^N$ and its mass and the
flat distance between integral currents was first defined by
Federer and Fleming in 1960 \cite{FF}.      
The Federer-Fleming notion of mass is a
weighted volume defined for integral currents (which are weighted oriented submanifolds built from countable collections of
Lipschitz submanifolds).   It is unrelated to ADM mass.

Any embedded $n$-submanifold of $\E^N$,  $\varphi: M^n\to \E^N$, can be thought of as a functional $T$ on  $n$-forms (\textit{i.e.\ }a current) as follows. For each $n$-form $\omega$ of compact support, 
\be
T(\omega):=\varphi_\#\lbrack M \rbrack \omega=\int_M\varphi^*\omega.
\ee
This concept can be extended to 
weighted oriented submanifolds built from countable collections of
Lipschitz submanifolds $\varphi_i: A_i\subset \E^n \to \E^N$
with integer weights $a_i\in \Z$ to define an
{\em integer rectifiable current}:
\be
T(\omega):=\sum_{i=1}^\infty a_i {\varphi_i}_{\#}\lbrack A_i \rbrack \omega=\sum_{i=1 }^{\infty} a_i\int_{A_i}\varphi_i^*\omega.
\ee
The boundary of a current is defined by 
\be
\partial T (\omega) = T(d\omega)
\ee
so that in particular for a smooth submanifold with boundary:
\be
\partial \lbrack M \rbrack= \lbrack \partial M \rbrack.
\ee
An {\em integral current} 
is an integer rectifiable current whose boundary
is also an integer rectifiable current. 
They denote the space of $n$-dimensional integral currents in $\E^N$ to be $\intcurr_n(\E^N)$.
They include the ${\bf{0}}$ current whose action on any form
satisfies ${\bf{0}}(\omega)=0$.

Given $T_1, T_2 \in \intcurr_n(\E^N)$ and an open subset $O\subset \E^N$,  the flat distance between $T_1$ and $T_2$ in $O$  is defined to be 
\be
	d_{F_O}(T_1, T_2)= \inf\left\{ \mass(A) + \mass(B) : T_1-T_2 = A + \partial B \mbox{ \rm{in} } O\right\}
\ee
where the infimum is taken over all $A\in \intcurr_n(O)$
and all $B\in \intcurr_{n+1}(O)$, and $\mass$ is the mass 
of each of these integral currents in $O$.  This is not Federer-Fleming's
notation but we use this because it is simpler to extend this
notation.

Federer and Fleming   proved a compactness theorem
stating that if $\mass(T_i)\le V_0$, $\mass(\partial T_i)\le A_0$,
and $\spt T_i \subset K$ compact, then a subsequence of $T_i$
converges in the weak and flat sense to an integral current of the same dimension (possibly the $\bf{0}$ current).   This theorem is one of the foundational theorems
of the field of Geometric Measure Theory.   

The flat distance is an extrinsic notion, not an intrinsic one.
For example, if we consider the graphs:
\be \label{ext-ex}
\left\{(x,f_k(x)): x\in [0,\pi]\right\}\in \E^2 
\ee
with $f_k(x)$ piecewise linear with slope $\pm 1$
connecting the points
\be
(0,0), (1/(2k),1/(2k)), (2/(2k),0), (3/(2k), 1/(2k)),..., (1,0)
\ee
then we have corresponding integral currents $T_k$ of weight $1$
with
\be
\mass(T_k)=\sqrt{2}
\ee
and $d_F(T_k, T_\infty)\to 0$ where $T_\infty$ is the current
corresponding to the graph of $f_\infty$ identically equal to $0$.
This can be seen by taking $A_k=0$ and $B_k$ to be the sum of 
the $2$ dimensional triangular regions lying between the graphs of 
$f_k$ and $f$.   Observe that
\be
\mass(T_\infty)=1.
\ee
In fact, Federer-Fleming proved lower semicontinuity of mass
\cite{FF}:
\be
\liminf_{j\to\infty}\mass(T_j)\ge \mass(T_\infty).
\ee
However, note that in this example the intrinsic geometry of each $T_k$ is that of a line segment of length $\sqrt{2}$, while the limit space $T_\infty$ is a line segment of length $1$.

\subsection{Review of Ambrosio-Kirchheim Integral Currents}

In \cite{AK}, Ambrosio and Kirchheim extended the notion of integral currents on $\E^N$ to integral currents on a complete metric space $Z$ denoted $\intcurr_n(Z)$.   Their notion of a current $T$ acts on
$n+1$ tuples of Lipschitz functions $(f, \pi_1,...,\pi_n)$
rather than differential forms, so that a rectifiable current is
defined by a countable collection of bi-Lipschitz charts 
$\psi_i: A_i \to Z$ (where $A_i$ are Borel sets
in $\mathbb{E}^n$)
as follows
\be
T(f, \pi_1,..., \pi_n)=\sum_{i=1}^\infty a_i \psi_{i\#}\lbrack A_i \rbrack
(f, \pi_1,..., \pi_n)
\ee
where the push forward is defined
\be
{\psi_i}_{\#}\lbrack A_i \rbrack (f, \pi_1, \dots, \pi_n) 
= \int_{A_i} f\circ \psi_i \,\,
d(\pi_1\circ\psi_i) \wedge \cdots \wedge d(\pi_n\circ\psi_i).
\ee

Ambrosio-Kirchheim define mass in a more complicated way than Federer-Fleming so that they are able to prove lower semicontinuity of mass.  They prove
the following useful relationship 
between mass and Hausdorff measure
for currents with weight $1$:
\be
C_n\mathcal{H}_n(\set T) \le \mass(T) \le C'_n  \mathcal{H}_n(\set T)
\ee
where $C_n, C'_n$ are precise dimension dependent constants
and $\set(T)$ is the collection of points of positive density with
respect to $T$.   In addition if $T$ is an $n$ dimensional integral current on
$(Z,d)$ and we rescale $d$ by $\lambda>0$ then
 \be
\mass_{(Z,\lambda d)}(T)=\lambda^n \mass_{(Z,d)}(T).
\ee 
More generally, if $d'\ge d$ then
\be
\mass_{(Z,d')}(T)\ge \mass_{(Z,d)}(T).
\ee 

They define boundary:
\be
\partial T(f, \pi_1,...,\pi_n) = T(1, f, \pi_1,...,\pi_n).
\ee
The space of integral currents, denoted $\intcurr_n(Z)$, is
the collection of integer rectifiable currents whose boundaries 
are integer rectifiable.  Again there is the $\bf{0}$ integral current
in each dimension.  The notion of flat distance naturally extends,
which we denote as $d_F^Z(T_1, T_2)$. 
They generalized Federer and Fleming's compactness theorem to this setting replacing $O$ with the requirement that $Z$ is compact.

\subsection{Gromov-Hausdorff Convergence}
In order to define Gromov-Hausdorff and intrinsic flat
convergence we need the following notion:

\begin{defn}
A map $\varphi: X \to Y$ between metric spaces, $(X, d_X)$ and $(Y, d_Y)$,
is a metric isometric embedding iff it is distance preserving:
\be
d_Y(\varphi(x_1), \varphi(x_2)) = d_X(x_1, x_2) \qquad \forall x_1, x_2 \in X.
\ee
\end{defn}

It is of crucial importance that this does not agree with 
the Riemannian notion of an isometric embedding.
See \cite{LeeSormani1} for a discussion of the distinction.

Although our main results do not directly involve Gromov-Hausdorff convergence, it is applied significantly within the paper.

\begin{defn}[Gromov]\label{defn-GH} 
The Gromov-Hausdorff distance between two 
compact metric spaces $\left(X, d_X\right)$ and $\left(Y, d_Y\right)$
is defined as
\be \label{eqn-GH-def}
d_{GH}\left(X,Y\right) := \inf  \, d^Z_H\left(\varphi\left(X\right), \psi\left(Y\right)\right)
\ee
where the $\inf$ is taken over all complete metric space $Z$ and metric isometric embeddings $\varphi: X \to Z$ and $\psi:Y\to Z$.  The Hausdorff distance in $Z$ is defined as
\be
d_{H}^Z\left(A,B\right) = \inf\left\{ \epsilon>0: A \subset T_\epsilon\left(B\right) \textrm{ and } B \subset T_\epsilon\left(A\right)\right\}.
\ee
\end{defn}

Gromov proved that this is indeed a distance on compact metric spaces
in the sense that $d_{GH}\left(X,Y\right)=0$
iff there is an isometry between $X$ and $Y$ \cite{Gromov-metric}.   
He also proved the following embedding theorem in \cite{Gromov-poly}:

\begin{thm}[Gromov] \label{Gromov-Z}
If a sequence of compact metric spaces, $X_j$,
converges in the Gromov-Hausdorff sense to a compact metric space $X_\infty$, 
\be
X_j \GHto X_\infty
\ee
then in fact there is a compact metric space, $Z$, and isometric embeddings $\varphi_j: X_j \to Z$ for $j\in \left\{1,2,...,\infty\right\}$ such that
\be
d_H^Z\left(\varphi_j(X_j),\varphi_\infty(X_\infty)\right) \to 0.
\ee
\end{thm}

This theorem allows one to define converging sequences of
points: 

\begin{defn}\label{Gromov-points}
One says that $x_j \in X_j$ converges to $x_\infty \in X_\infty$,
if there is a common space $Z$ as in Theorem~\ref{Gromov-Z}
such that $\varphi_j(x_j) \to \varphi_\infty(x)$ as points in
$Z$. 
\end{defn}

One can apply Theorem~\ref{Gromov-Z} to see that for any $x_\infty\in X_\infty$ there exists $x_j\in X_j$ converging to $x_\infty$ in this sense.   Theorem~\ref{Gromov-Z} also implies the following Gromov-Hausdorff Bolzano-Weierstrass Theorem:

\begin{thm}[Gromov] \label{Gromov-B-W}
Given compact metric spaces, $X_j \GHto X_\infty$, and $x_j\in X_j$, there is a 
subsequence, also denoted $x_j$, that
converges to some point $x_\infty\in X_\infty$ in the sense described above.
\end{thm}

Gromov's embedding theorem can also be applied in combination with other extension theorems to obtain the following Gromov-Hausdorff Arzela-Ascoli Theorem.   See also the appendix of a paper of Grove-Petersen \cite{Grove-Petersen} for a detailed proof and prior work of the last named author for a more general statement \cite{Sor-cosmos}.

\begin{thm}[Gromov, Grove-Petersen] \label{Gromov-Arz-Asc}
Given compact metric spaces $X_j \GHto X_\infty$ and $Y_j \GHto Y_\infty$
and equicontinuous functions $f_j: X_j \to Y_j$ in the sense that
\be
\forall \epsilon>0 \,\,\exists \delta_\epsilon>0\textrm{ such that }
d_{X_j}(x,x')< \delta_\epsilon \,\implies \,
d_{Y_j}(f_j(x), f_j(x'))\le \epsilon,
\ee
there exists a subsequence, also denoted $f_j: X_j \to Y_j$,
which converges to a continuous function $f_\infty: X_\infty \to Y_\infty$
in the sense that there exists common compact metric spaces $Z, W,$
and metric 
isometric embeddings $\varphi_j: X_j \to Z$, $\psi_j: Y_j \to W$
such that
\be
\lim_{j\to \infty} \psi_j(f_j(x_j)) = \psi_\infty(f_\infty(x_\infty))
 \textrm{ whenever } \lim_{j\to \infty}\varphi_j(x_j)=\varphi_\infty(x_\infty). 
\ee
Furthermore, if $\Lip(f_j) \le K$ then $\Lip(f_\infty)\le K$.
\end{thm}
Examples of Gromov-Hausdorff limits of rotationally
symmetric manifolds with nonnegative scalar curvature
and ADM mass converging to $0$ are provided in
work of the second and third authors \cite{LeeSormani1}.
In particular, they need not converge to Euclidean
space in the Gromov-Hausdorff sense.

\subsection{Intrinsic Flat Convergence}

In \cite{SorWen2}, the third named author and Wenger applied Ambrosio and Kirchheim's notion of an integral current to  define integral current spaces $(X,d,T)$, with $T\in \intcurr_n(\overline{X})$ and $\set(T)=X$ where $\overline{X}$ is the completion of $X$ and  $\set(T)$ is the set of positive density for $T$.  
This integral current structure, $T$, can be represented
by a collection of bi-Lipschitz charts 
\be
\psi_i: A_i \subset \E^n\too U_i\subset X
\ee
such that 
\be
\mathcal{H}^n\left(X \smallsetminus \bigcup_{i=1}^\infty \psi_i(A_i)\right)=0,
\ee
with integer valued Borel weight functions $\theta_i: A_i \to \Z$. 
So integral current spaces are countably $\mathcal{H}^n$ rectifiable
metric spaces endowed with oriented charts and integer weights. In particular, oriented Riemannian manifolds with finite volume can be regarded as integral
current spaces.
There is also the $\bf{0}$ integral current space in each dimension.    The $\bf{0}$ integral current space has current
structure $0$ and no metric space.  

Riemannian manifolds of finite volume are integral current spaces
where $(X,d)$ is the manifold with the intrinsic Riemannian distance
function defined using infimum over the lengths of curves lying within
the manifold.  The integral current structure $T$ is defined by
\be
T(f, \pi_1,...,\pi_n)=\int_M f\, d\pi_1\wedge\cdots \wedge d\pi_n.
\ee

Given an integral current space $M=(X,d,T)$, we can define
$\partial M=(\set(\partial T), d, \partial T)$.  Note that the boundary
is endowed with the restricted metric from the 
metric completion of the original space, $\bar X$, and that
its metric space is a subset of $\bar{X}$.   When $M$ is a 
Riemannian manifold with boundary, $\partial M$ is the
manifold boundary of $M$ endowed with the restricted metric.

Wenger and the third named author used Ambrosio-Kirchheim's notion of $\mass(T)$ and 
the push forward $\varphi_\#T$ to define the intrinsic flat distance
as follows.

\begin{defn}[\cite{SorWen2}]\label{IF-defn} 
Given two $n$-dimensional precompact integral current spaces 
$M_1=(X_1, d_1, T_1)$ and $M_2=(X_2, d_2,T_2)$, the 
intrinsic flat 
distance between the spaces is defined by
\be
d_{\mathcal{F}}\left(M_1, M_2\right)=\inf\left\{d_F^Z(\varphi_{1\#}T_1, \varphi_{2\#}T_2):\,\,\varphi_j: X_j \to Z \right\}
\ee
where the infimum is taken over all  complete metric
spaces $Z$ and all metric isometric embeddings $\varphi_j: X_j \to Z$:
\be
d_Z(\varphi_j(x), \varphi_j(x'))= d_{X_j}(x,x') \qquad \forall x,x' \in X_j.
\ee
\end{defn}

Note the similarity to the definition of the Gromov-Hausdorff distance with the distinction being that the Hausdorff distance between subsets of compact metric spaces, $Z$,
has been replaced by the flat distance between integral
currents in complete metric spaces, $Z$.   Two precompact integral current spaces, $M_i$,
have $d_{\mathcal{F}}\left(M_1, M_2\right)=0$ iff
there is a current preserving isometry between the spaces.   
If $M_i$ are Riemannian manifolds with weight $1$, this means  
there is an orientation preserving isometry between them.

\begin{example}\label{example:intrinsic-flat-topology}
The intrinsic flat distance is an intrinsic notion not
an extrinsic notion.   In the prior section we observed
that graphs $f_j: [0, 1]\to \E^2$ defined in (\ref{ext-ex}) converge
in the flat sense to the graph of $f_\infty$ which is identically $0$.
These graphs are intrinsically isometric to Riemannian manifolds $M_j =[0,\sqrt{2}]$, as seen using the  embeddings $\Psi_j(s)=(s/\sqrt{2}, f_j(s/\sqrt{2}))$.   Since the $\Psi_j$ are not metric isometric embeddings, we cannot use
these embeddings into $Z=\E^2$ to determine the intrinsic flat
limit of the $M_j$.  
The intrinsic flat limit is $M_\infty=[0,\sqrt{2}]$ which can be seen
clearly just by taking the identity map into $Z=[0,\sqrt{2}]$
and taking $A_j=0$ and $B_j=0$.     
\end{example}

In \cite{LeeSormani1}, the last named authors prove that
rotationally symmetric manifolds with nonnegative scalar curvature
whose ADM mass converges to $0$ that have no closed
interior minimal surfaces converge to Euclidean space in
the pointed intrinsic flat sense.   The proof there explicitly
constructs a sequence of metric spaces $Z_j$ and
integral currents $A_j$ and $B_j$ in $Z_j$ such that 
$\partial B_j +A_j=\varphi_{1\#}T_1-\varphi_{2\#}T_2$.
In the Appendix we prove a new theorem which allows one
to determine the intrinsic flat limits of certain sequences of
integral current spaces using explicit $Z_j$.


\subsection{Review of Theorems about Intrinsic Flat Convergence}
In this paper we will prove our results by applying the
following key theorems.

Wenger and the third named author proved
the following embedding theorem for convergent sequences of integral current spaces in \cite{SorWen2}.    The theorem
applies Ambrosio-Kirchheim's lower semicontinuity of mass
\cite{AK}.

\begin{thm}[\cite{SorWen2}]\label{converge}
If a sequence of 
integral current spaces,
$M_{j}=\left(X_j, d_j, T_j\right)$, converges  in the intrinsic flat sense to an 
integral current space,
 $M_\infty=\left(X_\infty,d_\infty,T_\infty\right)$, then
there is a separable
complete metric space, $Z$, and metric isometric embeddings  $\varphi_j: X_j \to Z$ such that
$\varphi_{j\#}T_j$ flat converges to $\varphi_{\infty\#} T_\infty$ in $Z$
and thus converge weakly as well.   

In particular we have lower semicontinuity
of mass
\be \label{semicont-mass}
\mass(T_\infty)\le \liminf_{j\to\infty}\mass(T_j).
\ee
\end{thm}

Wenger proved the following compactness theorem (stated in the language of integral current
spaces here):

\begin{thm}[Wenger~\cite{Wenger-compactness}]\label{thm-Wenger-compactness}  
Let $V_0, A_0, D > 0$ and let $M_j=(X_j, d_j, T_j)$
be a sequence of integral current spaces of the same dimension
such that
\be \label{eqn-compact-1}
\mass(T_j)\le V_0 \textrm{ and } \mass(\partial T_j) \le A_0
\ee 
and 
\be \label{eqn-compact-2}
\diam(X_j) \le D.
\ee
Then there exists a subsequence of $M_j$ (still denoted $M_j$) and an
 integral current space, $M_\infty$,
of the same dimension
(possibly the $\bf{0}$ space) such that
\be
\lim_{j\to\infty}d_{\mathcal{F}}\left(M_j, M_\infty \right)=0.
\ee
\end{thm}

Additional key definitions and theorems needed in this
paper were introduced by the third named author in \cite{Sormani-AA}.

\begin{defn}[\cite{Sormani-AA}]
Using the notation of Theorem~\ref{converge}, we say that $p_j\in X_j$ converges to $p\in X_\infty$ if
\be
\lim_{j\to\infty} \varphi_j(p_j)=\varphi_\infty(p) \in Z,
\ee
and we say $p_j$ disappears if
\be
\lim_{j\to\infty} \varphi_j(p_j)=z \in Z
\ee
but $z\notin \varphi_\infty(X_\infty)$.  In this definition
we have already chosen a sequence of embeddings as
in Theorem~\ref{converge}.
\end{defn}

\begin{rmrk}\label{thin-wells}
Note that if a sequence $(X_j, d_j, T_j)\Fto(X_\infty,d_\infty,T_\infty)$
and $(X_j, d_j)\GHto (X_\infty, d_\infty)$ then by the 
Gromov Embedding Theorem we can choose the same
sequence of embeddings and same target space $Z$ for
both notions of convergence.  By the Gromov-Hausdorff 
Bolzano-Weierstrass Theorem, no points disappear.  The possibility of disappearance
occurs when the intrinsic flat limit is smaller than the 
Gromov-Hausdorff limit due to either the cancellation or
collapse of certain regions in the sequence.   This occurs for
example in sequences of rotationally symmetric manifolds with increasingly thin gravity wells studied
in work of the second and third authors \cite{LeeSormani1}
where the Gromov-Hausdorff limit of the sequence is
Euclidean space with a line segment attached and the intrinsic
flat limit is Euclidean space.   The points in the thin wells
disappeared under intrinsic flat convergence but have limits
lying on the line segment in the Gromov-Hausdorff limit.
It is also possible that a sequence has no Gromov-Hausdorff
limit at all as in the Ilmanen Example (cf. \cite{SorWen2})
in which case many points disappear in the limit.
 \end{rmrk}

\begin{lem}[{\cite[Lemma 4.1]{Sormani-AA}}]\label{ball-converge} 
Suppose $M_i=(X_i, d_i, T_i)$ are integral current spaces 
which converge in the intrinsic flat sense to a 
nonzero integral current space 
$M_\infty=(X_\infty, d_\infty, T_\infty)$. If $p_j\to p_\infty\in \overline{X}_\infty$, then there exists a subsequence  such that 
for almost every $r>0$, 
\be
S(p_j, r) \Fto S(p_\infty, r) 
\ee
where  $S(p_i,\rho):=(\overline{B(p_i,\rho)}, d_i, T_i\rstr \overline{B(p_i,\rho)})$. 
\end{lem}

Due to the possibility of disappearing points under intrinsic flat
convergence, one does not have such strong Bolzano-Weierstrass
and Arzela-Ascoli Theorems as for Gromov-Hausdorff convergence.  There are a few proven
by the third named author in \cite{Sormani-AA}.  The following theorem
is particularly useful for this paper.

\begin{thm}[{\cite[Theorem 6.1]{Sormani-AA}}]\label{Flat-Arz-Asc} 
Fix $K>0$.
Suppose that $M_i=(X_i, d_i, T_i)$ is a sequence of integral current spaces with
$M_i \Fto M_\infty$
and that $\Psi_i: X_i \to W$ are Lipschitz maps into
a compact metric space $W$ with 
\be
\Lip(\Psi_i)\le K.
\ee
Then a subsequence of $\Psi_i$ (still denoted $\Psi_i$) converges to a Lipschitz map
$\Psi_\infty: X_\infty \to W$ with 
\be
\Lip(\Psi_\infty)\le K.
\ee  More
specifically, there exist metric isometric embeddings 
of the subsequence, $\varphi_i: X_i \to Z$,
such that $d_F^Z(\varphi_{i\#} T_i , \varphi_{\infty\#} T_\infty)\to 0$
and for any sequence $p_i\in X_i$ converging to $p\in X_\infty$,
one has converging images,
\be\label{Fip}
\lim_{i\to\infty} \Psi_i(p_i) = \Psi_\infty(p).
\ee
\end{thm}


\subsection{Graphical Hypersurfaces of Euclidean Space} \label{section:AF-graph}

We first recall that the spatial $n$-dimensional Schwarzschild manifold (with boundary) of mass $m>0$ can be isometrically embedded into $\mathbb{E}^{n+1}$ as the graph of the function $S_m(|x|)$ where 
\be\label{S_m-eqn}
	S_m(r) =\left\{
	\begin{array}{ll}
	 \sqrt{ 8m (r - 2m)}  &\mbox{ for } n =3\\
	 \sqrt{2m} \log \left( \frac{r}{\sqrt{2m}} + \sqrt{\frac{r^2}{2m} -1} \right) &\mbox{ for } n =4\\
	S_\infty + O(r^{2-\frac{n}{2}}) &\mbox{ for } n \ge 5,
	\end{array}\right.
\ee
for some constant $S_\infty$ depending on $n$ and $m$. The function $S_m$ arises from solving the ODE for a rotationally symmetric graph with zero scalar curvature.

For asymptotically flat graphs, one can define the ADM mass as follows. 
\begin{defn}[\cite{Lam-graph}]
Let $f$ be a $C^2$ function defined on an exterior region of $\E^n$. The \emph{ADM mass} of the graph of $f$ is defined by
\begin{align} \label{eq:mass} 
	m &= \frac{1}{2(n-1)  \omega_{n-1}}  \lim_{r\rightarrow \infty}\int_{|x|=r} \frac{1}{ 1 + |D f|^2 } \sum_{i,j=1}^n  (f_{ii} f_j - f_{ij} f_i)  \frac{x^j}{|x|}  \, d\mathcal{H}^{n-1},
\end{align}
where $\omega_{n-1}$ is the volume of the unit $(n-1)$-sphere
and $D f$ is the gradient of $f$ as a function on $\E^n$.
\end{defn}
The above definition coincides with the usual definition of the ADM mass under additional assumptions on the fall-off rates of $|D f|$ and $|D^2 f|$, see \cite{Lam-graph, Huang-Wu:2013}. 

\begin{thm}[\cite{Reilly:1973}]
Let $f$ be a $C^2$ function defined on an open subset of $\mathbb{E}^n$. Then the scalar curvature of the graph of $f$ is 
\[
R=\sum_{j=1}^n\frac{\partial}{\partial x_j} \left[\sum_{i=1}^n \left( \frac{f_{ii} f_j - f_{ij} f_i}{1+|D f|^2}\right)\right].
\]
\end{thm}
The above formula of scalar curvature is closely related to the definition of the ADM mass~\eqref{eq:mass}. In fact, let  $\Omega_h$ be a bounded subset of $\mathbb{E}^n$ such that $\partial \Omega_h = f^{-1}(h) :=\Sigma_h$.  Combining this theorem with the divergence theorem, and using the definition of ADM mass above, Lam~\cite{Lam-graph}  obtained that for any regular value $h$ of $f$,
\begin{align} \label{eq:Lam-mass}
	2(n-1) \omega_{n-1} m = \int_{\mathbb{R}^n \setminus \Omega_h} R\, dx + \int_{\Sigma_h} \frac{|D f|^2 }{1+ |Df|^2} H_{\Sigma_h} d \mathcal{H}^{n-1},
\end{align}
where  $H_{ \Sigma_h}$ is the mean curvature of $\Sigma_h$ in the hyperplane $\left\{x^{n+1}=h\right\}$ (with respect to inward pointing normal). For an entire graph, by setting $\Omega_h = \emptyset$, it immediately implies that the ADM mass is nonnegative~\cite{Lam-graph}. We also note that under the nonnegative scalar curvature assumption, the ADM mass always exists, though it may be infinite.

The first named author and Wu~\cite{Huang-Wu:2013, Huang-Wu-Penrose} proved that under the nonnegative scalar curvature assumption, the mean curvature of an asymptotically flat hypersurface in $\mathbb{E}^{n+1}$ (complete or with a minimal boundary) has a sign and that $H_{\Sigma_h}$ is nonnegative for almost every regular value $h$. They further concluded that if the ADM mass is zero, then there is no regular value $h$ and thus the hypersurface must be a hyperplane.

Based on previous work, the first two named authors studied the weighted mean curvature integral appearing in~\eqref{eq:Lam-mass}, which may be regarded as a quasi-local mass  for level sets. Together with the Minkowski inequality, they were able to prove a differential inequality for the volume functions of the level sets, as long as the volume function is greater than $\omega_{n-1}(2m)^{\frac{n-1}{n-2}}$. The differential inequality guarantees that the volumes of the level sets grow as fast as they do for Schwarzschild spaces of comparable mass. It is natural to define the height of the level set whose volume realizes this volume, but technically there might be no level set with this volume and the differential inequality may not be differentiable at this volume, so we define the height $h_0$ as follows.  
  
\begin{defn}[{\cite[Definition 3.7]{Huang-Lee-Graph}}]\label{definition:h_0}
Let $n\ge3$, $r_0, \gamma, D>0$, $\alpha<0$, and $r\ge r_0$. Let $M\in\Gr$ have ADM mass $m>0$. We may choose $\Psi$ and $f$ so that the graph has upward pointing mean curvature (\cite{Huang-Wu:2013, Huang-Wu-Penrose}). We define the height
\[
	h_0 =\sup \left\{ h: \mathcal{H}^{n-1}\left(f^{-1}(h)\right) \le 2\omega_{n-1}(2m)^{\frac{n-1}{n-2}} \mbox{ for regular value $h$}\right\}.
\]
\end{defn}
The set of $h$ with the desired property in Definition~\ref{definition:h_0} is non-empty (by the Penrose inequality~\cite{Lam-graph} in the case of minimal boundary). Note $h_0$ always exists and is finite because the level sets $f^{-1}(h)$ move outward as $h$ increases and the volume function is monotone nondecreasing in $h$~\cite[Proof of Lemma 3.3]{Huang-Lee-Graph}.

\begin{thm}[\cite{Huang-Lee-Graph}] \label{HL-Slab}
Let $n\ge3$, $r_0, \gamma, D>0$, $\alpha<0$, and $r\ge r_0$. We  normalize $\Psi(M)$ and $f$ so  that the graph has upward pointing mean curvature and that $h_0=0$. For any $\epsilon>0$, there exists a $\delta=\delta(\epsilon, n, \gamma, \alpha, r)>0$ such that
 if $M\in\Gr$ has ADM mass less than 
$\delta$, then  
\be
f(x)<\epsilon\,\text{ for all }\,|x|<r,
 \ee
  or in other words,
$\Psi(\Omega(r))=\Psi(M)\cap (B(r)\times\rr)$ lies below the plane $\E^n\times\left\{\epsilon\right\}$.
\end{thm}
Although this statement does not appear in \cite{Huang-Lee-Graph}, it is a direct consequence of Theorems 3.10 and 4.5 from \cite{Huang-Lee-Graph}. These two theorems were the main ingredients in the proof of following stability theorem with respect to the flat distance. 

\begin{thm}[{\cite[Theorems 5.2 and 5.3]{Huang-Lee-Graph}}] \label{HL-Flat}
Let $n\ge3$, $r_0, \gamma, D>0$, $\alpha<0$, and $r\ge r_0$. We vertically normalize $\Psi(M)$ such that $h_0 = 0$ for all $M$ in $\Gr$. For any $\epsilon>0$, there exists a $\delta=\delta(\epsilon, n, \gamma, D, \alpha, r)>0$ such that
  if $M\in\Gr$ has ADM mass 
less than $\delta$, then
\be
 d_{F_{B(r)}}(\, \Psi(M)\, , \, \E^n\times\left\{0\right\} \, )  \,<\,  \epsilon,
 \ee
 where $B(r)$ is the Euclidean ball of radius $r$ centered at the origin  in $\mathbb{E}^{n+1}$. 
  \end{thm}
The above two theorems do not actually require all of the hypotheses used to define $\Gr$, but we state the theorem this way for simplicity and for ease of comparison with Theorem \ref{thm-main}.

\begin{rmrk}
Note that the normalization hypothesis $h_0=0$ is not needed in Theorem~\ref{thm-main} and Theorem~\ref{thm-pted} because the statements in both theorems are invariant under vertical translations of $f$. 
\end{rmrk}


\section{Existence of a Limit}\label{sect-lim}

We now begin our proof of Theorem \ref{thm-main}. We fix $n\ge3$, $r_0, \gamma, D>0$,  $\alpha<0$, and $r\ge r_0$ once and for all. Given $M\in \Gr$, we define 
\be
\Omega(r)=\Psi^{-1}( B(r)\times\rr)
\ee 
as in the statement of Theorem \ref{thm-main}.   We also
define
\be
\Sigma(r)=\partial\Omega(r)\smallsetminus \partial M = \Psi^{-1}(\partial B(r)\times\rr).
\ee
All these spaces are endowed with the restricted metric from
$M$.  

 Our first task is to extract a limit. That is, we prove Theorem~\ref{exists-lim} that the family of $\Omega(r)$ coming from $\Gr$ is precompact in the intrinsic flat topology.  


\begin{thm}\label{exists-lim}
Let $r$ be fixed. Given a sequence $M_j \in \Gr$  there is a subsequence (still denoted $M_j$) and an integral current space
\be
\Omega_\infty(r)=(X_\infty, d_\infty, T_\infty)
\ee
such that
\be
\lim_{j\to\infty}d_{\mathcal{F}}\left(\Omega_j(r), \Omega_\infty(r)\right) = 0.
\ee
There also exists a 1-Lipschitz map
\be
\Psi_\infty: \Omega_\infty(r) \too \overline{B(r)}\times \R \subset \E^{n+1}
\ee
which is a limit of $\Psi_j$ as in Theorem~\ref{Flat-Arz-Asc}.
\end{thm}

\begin{rmrk}
Note that this lemma does not require ADM mass to converge to $0$.
\end{rmrk}

\begin{rmrk}
A stronger compactness result was proven in the
rotationally symmetric case by the third named author with LeFloch \cite{LeFloch-Sormani-1}
for $\Omega_j(r_0)$ of uniformly bounded
depth, $\Depth(\Omega_j(r),\Sigma_j(r_0))\le D_0$,
and Hawking mass, $m_H(\Sigma_j(r_0))\le m_0$,
that lie within
symmetric manifolds, $M_j$, with nonnegative
scalar curvature that have no closed interior minimal surfaces.
Specifically they prove there 
is a subsequence  
\be
\Omega_j(r_0)\Fto \Omega_\infty(r_0)
\ee
where $\Omega_\infty(r_0)$ is a rotationally symmetric
integral current space which has weakly nonnegative scalar
curvature and $m_H(\Sigma(r_0))\le m_0$.   In addition the
Hawking masses converge to the generalized Hawking mass
of the limit space and the limit space has generalized nonnegative
scalar curvature.  One key step in
that theorem is the proof that the limit space is not the $\bf{0}$
space.  In Theorem~\ref{exists-lim} we do not yet elliminate the
possibility that $\Omega_\infty(r)=\bf{0}$ nor do we prove
$\Omega_\infty(r)$ has curvature and Hawking mass bounds.
This would be an interesting question.
\end{rmrk}

\begin{proof}
We first check the hypotheses of Wenger's Compactness Theorem (cf. Theorem \ref{thm-Wenger-compactness}). Because of the gradient bound $|D f_j|\le \gamma$ for $|x|\ge r_0$, it follows that 
\be
\vol(\partial \Omega_j(r))\le \omega_{n-1}r^{n-1}\sqrt{1+\gamma^2}.
\ee
The gradient bound also means that the distance between any two points in $\Omega_j(r)\smallsetminus \Omega_j(r_0)$ is bounded by $\pi r \sqrt{1+\gamma^2}$. 

Since $\Depth(\Omega_j(r_0), \Sigma_j(r_0))\le D$ by Definition \ref{def:hypotheses}, 
it follows that
\be
\diam(\Omega_j(r)) \le 2D + \pi r \sqrt{1+\gamma^2}.
\ee
For the volume bound, we use the coarea formula to estimate
\begin{align}
\vol(\Omega_j(r))
&= \int_{B(r)\smallsetminus U_j} \sqrt{1+|D f_j|^2}\, d\mathcal{L}^n\\
&\le  \int_{B(r)\smallsetminus U_j} (1+|D f_j|) \, 	d\mathcal{L}^n \\
&\le \vol(B(r))+ \int_{-\infty}^\infty \mathcal{H}^{n-1}(f_j^{-1}(h)\cap B(r))\,dh.
\end{align}
 
\begin{figure}[here] 
   \centering
   \includegraphics[width=0.5\textwidth]{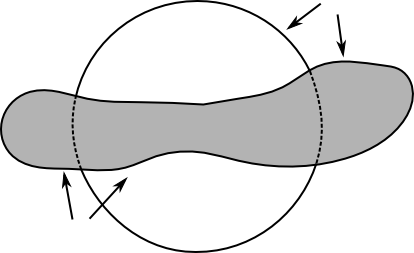}
         			\put(-115, 85){$B(r)$}
      			\put(-37, 110){$S'$}
			    \put(-148, 6){$S$}
			    \put(-100, 50){$E$}
   \caption{$S=\partial^*(E)$ and $S'=\partial^*(E\cup B(r))$}
   			\label{figure:outer-minimizing}
\end{figure}
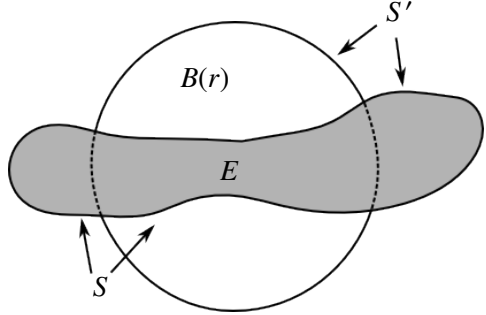
In order to estimate the volumes of the level sets, we claim that if $S\subset \E^n$ is an outward-minimizing hypersurface, then $\mathcal{H}^{n-1}(S\cap B(r))\le \mathcal{H}^{n-1}(\partial B(r))$. To see this, we set $S=\partial^*E$, where $\partial^*$ denotes the reduced boundary. 
 Let $S'=\partial^*(E\cup B(r))$. Then $\mathcal{H}^{n-1}(S)\le \mathcal{H}^{n-1}(S')$ by the outward-minimizing property of~$S$.    
 By removing the intersection,  $\mathcal{H}^{n-1}(S\smallsetminus S')\le \mathcal{H}^{n-1}(S'\smallsetminus S)$.    
 The claim then follows because $S\smallsetminus S'=S\cap B(r)$ and $S'\smallsetminus S\subset \partial B(r)$. See Figure~\ref{figure:outer-minimizing}.

Since almost every level set of $f$ is outward-minimizing, the claim shows that
 $\mathcal{H}^{n-1}(f_j^{-1}(h)\cap B(r))\le \mathcal{H}^{n-1}(\partial B(r))$ for almost every $h$.    Moreover, $f_j^{-1}(h)\cap B(r)$ must actually be empty away from an interval of length equal to $\diam\Omega_j(r)$. Thus
\be
\vol(\Omega_j(r))\le  \vol(B(r)) + \diam\Omega_j(r) \vol(\partial B(r))
\ee
and we already bounded $ \diam\Omega_j(r)$. Hence we can apply Wenger's Compactness Theorem (cf. Theorem~\ref{thm-Wenger-compactness}) to extract a subsequence of $\Omega_j(r)$ converging in the intrinsic flat sense.   As in his compactness theorem, this limit space may be $\bf{0}$.

The last conclusion involving $\Psi_\infty$ then follows immediately from Theorem \ref{Flat-Arz-Asc}, since each $\Psi_j$ is clearly  a distance non-increasing map into compact space $\overline{B(r)}\times [-D, D]$.
\end{proof}


\section{Geometric Estimates}\label{sect-geom}

We have shown in Theorem~\ref{exists-lim} that
\be
\Omega_j(r) \Fto \Omega_\infty(r)=(X_\infty, d_\infty, T_\infty).
\ee
Then immediately
\be
\partial \Omega_j(r) \Fto \partial \Omega_\infty(r)=(\set(\partial T_\infty), d_\infty, \partial T_\infty).
\ee 
In this section we use the fact that the
manifolds are graphs and have ADM mass converging to $0$
and apply prior work of the first two named
authors \cite{Huang-Lee-Graph} and Lam \cite{Lam-graph} to provide key
geometric estimates.
In Lemma~\ref{inner-gone}, we prove that the inner boundaries disappear and the
outer boundaries converge:
\be
\Sigma_j(r)\Fto \Sigma_\infty(r)=\partial \Omega_\infty(r).
\ee
In Lemmas~\ref{Vol-Bounds-} and~\ref{Vol-Bounds+}
we bound the volumes of $\Omega_j(r)$ from above and below: showing $\vol(\Omega_j(r))\to \vol(B(r))$.   In Lemma~\ref{thm-in-disk}
we prove that $\Psi_\infty(\Omega_\infty)$ lies in a Euclidean
disk.

\subsection{Inner boundaries disappear}

\begin{lem}\label{inner-gone}
Given the setup of Theorem \ref{exists-lim}, if we further assume that the ADM mass of $M_j$ converges to zero, then 
\be
\Sigma_j(r) \Fto \Sigma_\infty(r):=\partial \Omega_\infty(r)=(\set(\partial T_\infty), d_\infty, \partial T_\infty).
\ee
\end{lem}

\begin{proof}
Viewed as integral currents:
\be
\lbrack \partial\Omega_j(r) \rbrack-\lbrack \Sigma_j(r) \rbrack 
=\lbrack \partial M_j \rbrack.
\ee
By Lam's Penrose inequality \cite{Lam-graph}, we know that $\vol(\partial M_j) \le  \omega_{n-1}(2m_j)^{\frac{n-1}{n-2}}$, where $m_j$ is the ADM mass of $M_j$.    Thus
\be 
\lim_{j\to\infty}\mass\lbrack \partial M_j \rbrack=0.
\ee
By the definition of the intrinsic flat distance and the
fact that $\Sigma_j$ and $\partial \Omega_j$ are
endowed with the restricted metric from $M_j$, we have
\begin{eqnarray}
d_{\mathcal{F}}\left(\Sigma_j, \partial \Omega_j\right)
&\le& 
d_F^{M_j}(\lbrack \Sigma_j(r) \rbrack, \lbrack \partial\Omega_j(r))\rbrack)\\
&\le & \mass\lbrack \partial M_j \rbrack \, \to \, 0.
\end{eqnarray}
Since $\partial \Omega_j \Fto \partial \Omega_\infty$,
we have $\Sigma_j \Fto \partial \Omega_\infty$.
\end{proof}

\subsection{Volume Bounds}
In order to apply Theorem~\ref{HL-Slab}, throughout this section we adopt the convention that $\Psi$ and $f$ are chosen so that the graph has upward pointing mean curvature and that $\Psi(M)$ is vertically normalized such that $h_0=0$, where $h_0$ is defined by Definition~\ref{definition:h_0}. (See Section~\ref{section:AF-graph}.)  

By definition of $h_0$, for any regular value $h<h_0=0$, 
\be
	\mathcal{H}^{n-1}(f^{-1}(h)) = \vol(\Psi(M)\cap (\E^n\times\left\{h\right\})) < 2\omega_{n-1} (2m)^{\frac{n-1}{n-2}}, 
\ee
which immediately implies that for any $\epsilon>0$, there exists $\delta=\delta(\epsilon,n)$ such that if the ADM mass of $M\in \Gr$ is less then $\delta$, then 
\be \label{zero-height}
\mathcal{H}^{n-1}(f^{-1}(h)) < \epsilon.
\ee
Using this we can show that the part of $\Psi(M)$ lying under the plane $\E^n\times\left\{0\right\}$ must have small volume.

\begin{lem}\label{Vol-Bounds-}
For any $\epsilon>0$, there exists $\delta=\delta(\epsilon, n, \gamma, D, r)>0$ such that if $M\in \Gr$ has mass less than $\delta$, then 
\be
\vol(\Omega^-(r))<\epsilon
\ee
where $\Omega^-(r) := \Psi^{-1}(B(r)\times (-\infty, 0))$.
\end{lem}
\begin{proof}
We choose $\delta$ small enough so that \eqref{zero-height} holds.  Arguing as in the proof of Theorem~\ref{exists-lim}, we have
\begin{align}
\vol(\Omega^-(r))
&= \int_{f(x)<0} \sqrt{1+|D f|^2}\, d \mathcal{L}^n\\
&\le \mathcal{H}^n(f^{-1}(-\infty, 0)) +  \int_{-\infty}^{0} \mathcal{H}^{n-1}(f^{-1}(h)\cap B(r))\,dh.
\end{align}
By the isoperimetric inequality and \eqref{zero-height}, the first term is bounded by a constant times $\epsilon^{\frac{n}{n-1}}$. 

To estimate the second term, by \eqref{zero-height}, for almost every negative $h$, 
\be
\mathcal{H}^{n-1}(f^{-1}(h)\cap B(r))\le  \mathcal{H}^{n-1}(f^{-1}(h))<\epsilon.
\ee
 As in the proof of Theorem \ref{exists-lim}, it follows that 
 \be
 \int_{-\infty}^{0} \mathcal{H}^{n-1}(f^{-1}(h)\cap B(r))\,dh<\epsilon\diam(\Omega(r)),
 \ee
 and we know $\diam(\Omega(r))$ is bounded in term of $\gamma$, $D$, and $r$.
\end{proof}

Theorem \ref{HL-Slab} allows us to estimate the rest of the volume of $\Omega(r)$.

\begin{lem}\label{Vol-Bounds+}
For any $\epsilon>0$, there exists $\delta=\delta(\epsilon, n, r, \gamma, \alpha)>0$ such that
 if $M\in\Gr$ has  ADM mass less than $\delta$, then 
\be
\vol(\Omega^+(r))\le \vol(B(r))+\epsilon
\ee
where $\Omega^+(r) := \Psi^{-1}(B(r)\times [0,\infty))$.
\end{lem}

\begin{proof}
We choose $\delta$ small enough so  that Theorem \ref{HL-Slab}  holds.
As in the proof of Theorem \ref{exists-lim} we have
\be
\vol(\Omega^+(r))
\le \vol(B(r))+ \int_{0}^{\epsilon} \mathcal{H}^{n-1}(f^{-1}(h)\cap B(r))\,dh,
\ee
where the upper limit $\epsilon$ follows from Theorem \ref{HL-Slab}.
As in the proof of Theorem \ref{exists-lim}, for almost every $h$ we have
 $\mathcal{H}^{n-1}(f^{-1}(h)\cap B(r))\le  \vol(\partial B(r))$. Thus
\be
\vol(\Omega^+(r)) \le \vol(B(r))+\epsilon \vol(\partial B(r)).
\ee
\end{proof}

%


\begin{cor}\label{Vol-Bounds}
If $M_j\in \Gr$ is a sequence with masses approaching zero, then $\limsup_{j\to\infty} \vol(\Omega_j(r))\le \vol(B(r))$.
\end{cor}

\subsection{The Image of $\Psi_\infty$ Lies in a Disk}

Our goal is to show that $\Psi_\infty$ is an isometry from $\Omega_\infty(r)$ to $B(r)\times\left\{0\right\}$. The next lemma shows that the image falls in the correct place. However, technically we will only use the fact that boundary falls in the right place.

\begin{lem}\label{thm-in-disk}
Let $M_j$ be as in the statement of Theorem \ref{exists-lim}, and assume that the ADM mass of $M_j$ converges to zero. 
The 1-Lipschitz map
\be
\Psi_\infty: \Omega_\infty(r) \too \overline{B(r)}\times \R \subset \E^{n+1}
\ee
constructed in Theorem \ref{exists-lim} has image lying in the disk $\overline{B}(r)\times\left\{0\right\}$. In particular, the induced map on the boundary $\Sigma_\infty(r):=\partial\Omega_\infty(r)$ has image lying in $\partial B(r)\times\left\{0\right\}$.

\end{lem}
\begin{proof}
Recall that for any $p\in  \Omega_\infty(r)$, $\Psi_\infty(p)$ is the limit of $\Psi_j(p_j)$ for some sequence $p_j\in \Omega_j(r)$ converging to $p$. By Theorem \ref{HL-Slab}, we know that for $\epsilon>0$, the point $\Psi_j(p_j)\in\Psi_j(\Omega_j(r))$ lies below the height $\epsilon$ for large $j$. Thus $\Psi_\infty(p)$ lies in $\overline{{B}(r)}\times(-\infty,0]$.
Now suppose $\Psi_\infty(p)$ lies strictly below the height $0$. Then for any $\rho>0$ sufficiently small, the intrinsic ball $B(p_j, \rho)\subset M_j$ lies entirely inside $\Omega_j^-(r)$. By Lemma \ref{Vol-Bounds-}, we must have $\lim_{j\to\infty} \vol(B(p_j, \rho))=0$. But this contradicts the fact that $p_j\to p \in \Omega_\infty(r)$, by Lemma \ref{ball-converge}, for example.

For the second part of the lemma, consider $p\in  \partial\Omega_\infty(r)$. By Lemma \ref{inner-gone}, $\Psi_\infty(p)$ is the limit of $\Psi_j(p_j)$ for some sequence $p_j\in \Sigma_j(r)$ converging to $p$. Since each $\Psi_j(p_j)$ lies in  $\partial B(r)\times\rr$, the result now follows from the first part of the lemma.
 \end{proof}


\section{Bi-Lipschitz Map between $\Sigma(r)$ and $\partial B(r)$}
\label{sect-biLip}

In this section we prove that the outer boundaries $\Sigma_j(r)$
behave far better than $\Omega_j(r)$.   We already
know 
\be
\Omega_j(r) \Fto \Omega_\infty(r)=(X_\infty, d_\infty, T_\infty)
\ee
and by Lemma~\ref{inner-gone} we know
\be
\Sigma_j(r) \Fto \Sigma_\infty(r):=\partial \Omega_\infty(r)=(\set(\partial T_\infty), d_\infty, \partial T_\infty).
\ee

Now we prove far more:

\begin{lem} \label{biLip}
Assume the hypotheses of Lemma \ref{thm-in-disk}.
Then we have Gromov-Hausdorff convergence to the limit
\be\label{biLip-GH}
(\Sigma_j(r), d_j)
\GHto 
(\Sigma_\infty(r), d_\infty)
\ee
and the  map
\be
\Psi_\infty:\Sigma_\infty(r) \too \partial B(r)\times\left\{0\right\}
\ee
described in Lemma \ref{thm-in-disk} is a bi-Lipschitz map.  In particular, 
it follows that
  \be
\Psi_{\infty\#} (\partial T_\infty)= \lbrack \partial B(r)\times \left\{0\right\}\rbrack,
\ee
where  
$\lbrack \partial B(r)\times\left\{0\right\} \rbrack$ denotes the integral $(n-1)$-current in $\E^{n+1}$ corresponding to the $(n-1)$
dimensional submanifold $\partial B(r)\times \left\{0\right\}$.
 \end{lem}

 \begin{rmrk}\label{no-disappearing}
 By the Gromov-Hausdorff convergence we know that for
any sequence $p_j \in \Sigma_j(r)$, there
is a subsequence which converges to 
$p_\infty \in \Sigma_\infty(r)$.  In other words,
there are no disappearing sequences on the boundary.
There can be disappearing sequences inside $\Omega_j(r)$.
This can be seen in rotationally symmetric examples by choosing
 points in increasingly thin wells of uniform depth as in the
 work of the last two named authors \cite{LeeSormani1}.
 \end{rmrk}
 
 \begin{rmrk}
 This lemma strongly uses (\ref{cond3}) in the hypotheses
 on $\Omega(r)$.   Without this condition it is possible for 
 there to be sequences of $p_j \in \Sigma_j(r)$
 which disappear in the limit.     One may
 center the rotational symmetry of the previously
 mentioned example about a point in $\partial B(r)$
 if we do not require (\ref{cond3}).
 \end{rmrk}

\begin{proof}
Let $\pi$ be the obvious projection map from $\E^{n+1}$ to $\E^{n}\times\left\{0\right\}$, and
define the map
\be\label{phi-j-new}
\Phi_j: \partial B(r)\times\left\{0\right\} \too \Sigma_j(r)
\ee 
to be the inverse of the bijective map  
\be 
\pi\circ\Psi_j:\Sigma_j(r)\too \partial B(r)\times\left\{0\right\}.
\ee


We claim that the $\Phi_j$ have a uniformly bounded Lipschitz constant $\Gamma$.
For any $x_1, x_2\in \partial B(r)\times\left\{0\right\}$ with Euclidean distance $|x_1-x_2|\le \sqrt{2}r$, we can ``lift'' the chord joining $x_1$ to $x_2$ to a curve $c:[0,1]\too\Omega_j(r)$ joining $\Phi_j(x_1)$ to $\Phi_j(x_2)$ 
such that $\pi(\Psi_j(c(t)))= x_1(1-t)+x_2 t$. Note that this is possible because the chord joining $x_1$ to $x_2$ stays outside $B(r_0/2)\times\left\{0\right\}$ and $U\subset B(r_0/2)$ by Definition \ref{def:hypotheses}. 

For the same reason we can apply the gradient bound  in Definition \ref{def:hypotheses} to conclude that
\be
|c'(t)|\le |x_1-x_2|\sqrt{1+\gamma^2},
\ee
and consequently, 
\be
d_j(\Phi_j(x_1),\Phi_j(x_2))
\le |x_1-x_2|\sqrt{1+ \gamma^2}.
\ee
Now consider any pair $x_1, x_2\in \partial B(r)\times\left\{0\right\}$. There is a midpoint $x_3 \in \partial B(r)$ such that 
\be
|x_1-x_3|=|x_3-x_2|\le \sqrt{2}r.
\ee
 Then
 \begin{align}
\frac{ d_j(\Phi_j(x_1),\Phi_j(x_2))}
{|x_1-x_2|}
&\le  \frac{d_j(\Phi_j(x_1),\Phi_j(x_3))+d_j(\Phi_j(x_3),\Phi_j(x_2))
}{|x_1-x_3|}\\
&\le  \frac{   (|x_1-x_3|+  |x_3-x_2|)\sqrt{1+ \gamma^2}} {|x_1-x_3|} = 2\sqrt{1+\gamma^2}.
\end{align} 
Thus we have proven our claim:
\be\label{Lip-bound}
\Lip(\Phi_j) \le \Gamma,
\ee
where $\Gamma =2 \sqrt{1+ \gamma^2}$.

Next we will apply this uniform Lipschitz bound to prove 
Gromov-Hausdorff convergence by applying
Theorem~\ref{app-thm} from the appendix.   To apply this
theorem we need to view all our spaces as lying on a single
domain.

We consider the pullback metric $d_j'=\Phi_j^*d_j$ on $\partial B(r)\times \left\{0\right\}$, where $d_j$ is the metric on $\Omega_j(r)$.  So
we have 
\be\label{new-view}
(\partial \Omega_j(r), d_j, \lbrack \partial \Omega_j(r)\rbrack)
\isom(\partial B(r)\times \left\{0\right\}, d_j', \lbrack \partial B(r)\times \left\{0\right\}\rbrack)
\ee
via $\Phi_j$ as a current preserving isometry because we
are simply pulling back the metric.

We can now apply Theorem~\ref{app-thm} in the Appendix
because (\ref{Lip-bound}) implies that
\be
 1 \le \frac{ d_j'(x,y)}{d_{\E^{n+1}}(x,y)} \le  \Gamma
 \ee
for all $x, y\in\partial B(r)\times\left\{0\right\}$.  Thus 
there is a subsequence which we also denote $d'_j$ and
there exists a metric $d'_\infty=\lim_{j\to\infty}d'_j$ on $\partial B(r)\times\left\{0\right\}$, with the property that
\be \label{lim-Lip-1}
 1 \le \frac{ d'_\infty (x,y)}{d_{\E^{n+1}}(x,y)} \le  \Gamma,
 \ee
and such that our integral current spaces in (\ref{new-view}) converge (subsequentially) in
both the intrinsic flat and Gromov-Hausdorff
sense to
\be\label{new-lim}
(\partial B(r)\times \left\{0\right\}, d_\infty', \lbrack \partial B(r)\times \left\{0\right\}\rbrack).
\ee
However we know $\Sigma_j(r) \Fto \Sigma_\infty(r)$.
Thus there is a current preserving isometry so that
\be\label{new-lim-2}
( \Sigma_\infty(r), d_\infty, \partial T_\infty)
\isom(\partial B(r)\times \left\{0\right\}, d_\infty', \lbrack \partial B(r)\times \left\{0\right\}\rbrack).
\ee
So we have $\Sigma_j(r)\GHto \Sigma_\infty(r)$
for the subsequence.

Next we prove $\Psi_\infty$ is bi-Lipschitz.  Since it is defined
to be the limit of Lipschitz $1$ maps, we already know
$\Lip(\Psi_\infty)\le 1$.   We must construct the inverse map
and prove it is Lipschitz.

Since $\Phi_j: \partial B(r)\times \left\{0\right\} \too \Sigma_j(r)$
satisfy \eqref{Lip-bound}, we can apply the Gromov-Hausdorff
Arzela-Ascoli Theorem of Grove-Petersen \cite{Grove-Petersen}
to see that a further subsequence converges to 
\be
\Phi_\infty: \partial B(r)\times\left\{0\right\} \too \Sigma_\infty(r)
\ee
which also satisfies \eqref{Lip-bound}.   We need only
show $\Phi_\infty$ is the inverse of $\Psi_\infty$.

Since
\be
\Phi_j\circ\pi\circ \Psi_j=id: \Sigma_j(r) \to \Sigma_j(r)
\ee
and
\be
\pi \circ \Psi_j \circ \Phi_j= id: \partial B(r)\times\left\{0\right\}\to \partial B(r)\times\left\{0\right\}
\ee
we have
\be
\Phi_\infty\circ\pi\circ \Psi_\infty=id: \Sigma_\infty(r) \to \Sigma_\infty(r)
\ee
and
\be
\pi \circ \Psi_\infty \circ \Phi_\infty= id: \partial B(r)\times\left\{0\right\}\to \partial B(r)\times\left\{0\right\}.
\ee
Thus
\be
\pi \circ \Psi_\infty: \Sigma_\infty(r) \to \partial B(r)\times\left\{0\right\}
\ee
is the inverse of $\Phi_\infty$.   By Lemma \ref{thm-in-disk},
$\pi \circ \Psi_\infty=\Psi_\infty$.
\end{proof}


\section{Proof of Theorem~\ref{thm-main}}\label{sect-main}

We now complete the proof of Theorem~\ref{thm-main}:

\begin{proof}
 By Lemma \ref{biLip}, we know that 
  \be
\partial \lbrack B(r)\times \left\{0\right\}\rbrack=\Psi_{\infty\#} (\partial T_\infty)=\partial(\Psi_{\infty\#}T_\infty)
\ee
as an equality between $(n-1)$ currents in $\E^{n+1}$.

In the following computation, we use the minimizing property of the disk (among integral currents with the same boundary) in the first step, the fact that $\Lip(\Psi_\infty)\le 1$ in the second step, lower semicontinuity of mass (cf. Theorem \ref{converge}) in the third step, and Corollary \ref{Vol-Bounds} in the final step.

\begin{align}
\vol(B(r))&\le \mass(\Psi_{\infty\#}  T_\infty) \\
& \le \mass(T_\infty) \label{t35}\\
 &\le \liminf_{j\to\infty} \mass(T_j)\\
& =\liminf_{j\to\infty}\vol(\Omega_j(r)) \\
&= \vol(B(r)). 
\end{align}
 Equality in the first step implies that $\Psi_{\infty\#}T_\infty=\lbrack B(r)\times \left\{0\right\}\rbrack$, and then equality in the second inequality \eqref{t35} for a Lipschitz $1$ function implies that $\Psi_\infty:\Omega_\infty(r)\too  B(r)\times\left\{0\right\}$ must be an isometry.   In summary, we have shown that any sequence in $\Gr$ with the ADM mass converging to zero has a subsequence that converges to $B(r)\times\left\{0\right\}$ in the intrinsic flat distance.   The volume convergence follows from
Corollary~\ref{Vol-Bounds}.

To obtain the epsilon-delta formulation of Theorem~\ref{thm-main} from this is standard.
\end{proof}


\section{Proof of Pointed Convergence}\label{sect-pted}

We now turn to the proof of Theorem \ref{thm-pted}. Note that it is important that the points $p_j$ are not chosen arbitrarily.
It is easy to see that if $p_j$ is a sequence of disappearing points, the result will not hold, as can be seen in the example described in Remark \ref{thin-wells}.

\begin{proof}[Proof of Theorem \ref{thm-pted}]
Assume the hypotheses of Theorem \ref{thm-pted}. 
Fix any $R'> 0$.  It suffices to prove that for almost
every $R\in (0, R')$ we have (\ref{eq-pted}).

We claim that if $r=R' + r_0$ then $B_{p_j}(R)\subset \Omega_j(r)\subset M_j$.      
To see this, for each $q\in B_{p_j}(R)$, we have
\be
d_{\E^n}(\pi(\Psi_j(q)), \partial B(r_0)) \le 
 d_{M_j}(q, \Sigma_j(r_0)) < R<R'.
 \ee
Thus  $\pi(\Psi_j(q)) \subset B(r_0+R')\smallsetminus U$ and   
so $q\in \Omega(r_0+R')\subset M_j$.


By Theorem~\ref{thm-main} we know that 
\be
\lim_{j\to\infty}d_{\mathcal{F}}\left(\Omega_j(r) \subset M_j, B(r)\subset \E^n\right) =0
\ee
and by Lemma~\ref{biLip} we know that
\be
\lim_{j\to\infty}d_{GH}\left(\Sigma_j(r_0)\subset M_j, \partial B(r_0)\subset \E^n\right) =0.
\ee
So for any sequence of points $p_j \in \Sigma_j(r_0)$
there is a subsequence 
\be
p_{j_i}\to p_\infty\in \partial B(r_0)\subset B(r).
\ee
by the Gromov-Hausdorff Bolzano-Weierstrass Theorem
(cf. Theorem~\ref{Gromov-B-W}).
Then by the intrinsic flat ball convergence lemma (cf. Lemma \ref{ball-converge}),
$B_{p_{j_i}}(R) \Fto B_{p_\infty}(R)$.   Since $B_{p_\infty}(R)$
is isometric to a Euclidean ball of radius $R$ regardless of the
value of $p_\infty$,  we have (\ref{eq-pted}) and we are done
with the proof of pointed intrinsic flat convergence.

The volume convergence follows from the volume convergence
in Theorem~\ref{thm-main} and semicontinuity of volume as follows:
\begin{eqnarray*}
\liminf_{j\to\infty} \vol( B_{p_{j}}(R)) &=& \liminf_{j\to\infty} \mass( \lbrack B_{p_{j}}(R)\rbrack) \\
&\ge& \mass( \lbrack B_{p_{\infty}}(R)\rbrack) \\
&\ge&\vol(B(R)) \\
\limsup_{j\to\infty} \vol( B_{p_{j}}(R)) &=& \limsup_{j\to\infty} \mass(\lbrack B_{p_j}(R)\rbrack)\\
&\le& \limsup_{j\to\infty} \mass(\lbrack\Omega_j(r)\rbrack) -\liminf_{j\to\infty}\mass(\lbrack\Omega_j(r)\setminus  B_{p_j}(R)\rbrack)\\
&=& \limsup_{j\to\infty} \vol(\lbrack\Omega_j(r)\rbrack) -\liminf_{j\to\infty}\mass(\lbrack\Omega_j(r)\setminus  B_{p_j}(R)\rbrack)\\
&\le& \vol(B(r)) - \mass(\lbrack B(r)\setminus  B_{p_\infty}(R)\rbrack)\\
&=& \vol(B(r))-\vol(B(r))+\vol(B(R))\\
&=&\vol(B(R)).
\end{eqnarray*}
\end{proof}

\section{Appendix}\label{sect-app}

The following theorem concerning intrinsic flat limits of
integral current spaces with varying metrics may be applicable in
other settings as well.  The Gromov-Hausdorff part of this theorem was already
proven by Gromov in \cite{Gromov-metric} but with a completely
different proof in which the common metric space $Z$ is the
disjoint union.   The fact that one also obtains an intrinsic flat
limit which agrees with the Gromov-Hausdorff limit is new.

\begin{thm}\label{app-thm}
Fix a precompact $n$-dimensional integral current space $(X, d_0, T)$
without boundary (e.g. $\partial T=0$) and fix
$\lambda>0$.   Suppose that
$d_j$ are metrics on $X$ such that
\be\label{d_j}
\lambda \ge \frac{d_j(p,q)}{d_0(p,q)} \ge \frac{1}{\lambda}.
\ee
Then there exists a subsequence, also denoted $d_j$,
and a length metric $d_\infty$ satisfying (\ref{d_j}) such that
$d_j$ converges uniformly to $d_\infty$
\be\label{epsj}
\epsilon_j= \sup\left\{|d_j(p,q)-d_\infty(p,q)|:\,\, p,q\in X\right\} \to 0.
\ee 
Furthermore
\be\label{GHjlim}
\lim_{j\to \infty} d_{GH}\left((X, d_j), (X, d_\infty)\right) =0
\ee
and
\be\label{Fjlim}
\lim_{j\to \infty} d_{\mathcal{F}}\left((X, d_j,T), (X, d_\infty,T)\right) =0.
\ee
In particular, $(X, d_\infty, T)$ is an integral current space
and $\set(T)=X$ so there are no disappearing sequences of
points $x_j\in (X, d_j)$.

In fact we have
\be\label{GHj}
d_{GH}\left((X, d_j), (X, d_\infty)\right) \le 2\epsilon_j
\ee
and 
\be\label{Fj}
d_{\mathcal{F}}\left((X, d_j, T), (X, d_\infty, T)\right) \le
2^{(n+1)/2} \lambda^{n+1} 2\epsilon_j \mass_{(X,d_0)}(T).
\ee
\end{thm}

To prove this theorem we need a series of lemmas:

\begin{lem}\label{unif}
Under the hypothesis of Theorem~\ref{app-thm}, there exists a subsequence, also denoted $d_j$,
and a length metric $d_\infty$ satisfying (\ref{d_j}) such that
$d_j$ converges uniformly to $d_\infty$:
\be
\lim_{j\to \infty} \sup\left\{|d_j(p,q)-d_\infty(p,q)|:\,\, p,q\in X\right\} \to 0
\ee
and
 $(X, d_\infty, T)$ is an integral current space.
 \end{lem}

\begin{proof}
Observe that the functions $d_j$ may be extended to
the metric completion:
\be
d_j: \overline{X}\times \overline{X} \to [0, \diam_{d_j}(X)]\subset [0, \lambda\diam_{d_0}(X)].
\ee
By (\ref{d_j}) they are equicontinuous and so by the Arzela-Ascoli
Theorem they have a subsequence converging uniformly to a function
\be
d_\infty: \overline{X}\times \overline{X} \to [0, \lambda\diam_{d_0}(X)].
\ee
Taking the limit of \ref{d_j} we see that $d_\infty$
satisfies \red{(}\ref{d_j}\red{)} as well.   In particular $d_\infty$ is a 
metric on $X$.  

Furthermore the Ambrosio-Kirchheim
mass measure defined using $d_0$ and defined using 
$d_j$ may be related as follows:
\be
\lambda^n ||T||_0 \ge  ||T||_j \ge ||T||_0 \lambda^{-n}.
\ee 
Recall that $X=\set_0(T)\subset \bar{X}$ because $(X,d_0,T)$
is an integral current space (by the definition of integral current space).   
In general the set of positive density
also depends upon the metric just as the mass measure
done.   Here we have
\begin{eqnarray}
X=\set_0(T)
&=&\left\{p\in \bar{X}: \,\liminf_{r\to 0}\frac{||T||_0(B_p(r))}{r^n}>0\right\}\\
&=&\left\{p\in \bar{X}: \,\liminf_{r\to 0}\frac{||T||_j(B_p(r))}{r^n}>0\right\}\\
&=&\set_j(T)
\end{eqnarray}
and so $(X, d_j, T)$ is also an integral current space.   This is
true for all $j=1,2,...,\infty$.
\end{proof}

\begin{lem}\label{common-Z}
Given two metric spaces $(X, d_j)$ and $(X, d_\infty)$
there exists
a common metric space
\be
Z_j= [-\varepsilon_j, \varepsilon_j] \times X.
\ee
where
\be\label{epsj}
\epsilon_j= \sup\left\{|d_j(p,q)-d_\infty(p,q)|:\,\, p,q\in X\right\} 
\ee
with a metric $d'_j$ on $Z_j$ such that
\be \label{iso-}
d'_j((-\varepsilon_j,p), (-\varepsilon_j,q)) = d_j(p,q)
\ee
\be \label{iso+}
d'_j((\varepsilon_j,p), (\varepsilon_j,q)) = d_\infty(p,q).
\ee
Thus we have metric isometric embeddings
$\varphi_j:(X, d_j)\to (Z_j, d'_j)$ and
$\varphi_j':(X, d_\infty)\to (Z_j, d'_j)$ such that
\be
\varphi_j(p)=(-\varepsilon_j, p)
\textrm{ and } \varphi'_j(p)=(\varepsilon_j, p).
\ee
In addition, if $d_0, d_j$ satisfy  \eqref{d_j}, then 
\be\label{d0'}
d'_j(z_1,z_2) \le d'_0((t_1,p_1),(t_2, p_2)):= |t_1-t_2|+ \lambda d_0(p_1, p_2).
\ee
More precisely, we define $d'_j$ by 
\begin{eqnarray}
d'_j(z_1,z_2) &:=&\min\left\{d, d_-, d_+, d_{-+}, d_{+-}\right\} \textrm{ where }\\
d\,\,\,\,\,&=&d(z_1,z_2) \,\,\,\,\,\,= \,\,|t_1-t_2|+ \max\left\{d_j(p_1,p_2), d_\infty(p_1,p_2)\right\}\\
d_-\,\,\,&=&d_-(z_1,z_2)\,\,\,= \,\,|t_1+\epsilon_j|+|t_2+\epsilon_j| + d_j(p_1,p_2)\\
d_+\,\,\,&=& d_+(z_1,z_2)\,\,\,=\,\,|t_1-\epsilon_j|+|t_2-\epsilon_j| + d_\infty(p_1,p_2)\\
d_{-+}&=&d_{-+}(z_1,z_2)=\,\,\inf\left\{d_-(z_1,z)+d_+(z,z_2):\, z\in Z_j\right\}\\
    d_{+-}&=&d_{+-}(z_1,z_2)=\,\,\inf\left\{d_+(z_1,z)+d_-(z,z_2):\, z\in Z_j\right\}.
 \end{eqnarray}
\end{lem}

Note that $Z_j$ need not be a complete metric space,
even if $X$ is complete with respect to both metrics.
See Example~\ref{not-complete}.   However
we may always take the metric completion of $Z_j$ if we need
a complete metric space.   

Before proving this lemma we apply it to prove
Theorem~\ref{app-thm}:

\begin{proof}
First apply Lemmas~\ref{unif} and~\ref{common-Z}
and take the metric completion of $Z_j$ if it is not
yet complete.
Observe that 
\begin{eqnarray}
d'_j( (-\epsilon_j, p), (\epsilon_j, p)) 
&\le& d_-( (-\epsilon_j, p), (\epsilon_j, p))\\
&=& 0 + 2\epsilon_j + d_j(p,p) = 2\epsilon_j.
\end{eqnarray}
Thus
\be
d_H^{Z_j}(\varphi_j(X), \varphi_j'(X)) \le 2\epsilon_j
\ee
and we have (\ref{GHj}) which implies (\ref{GHjlim}).

To obtain (\ref{Fj}), we take 
$B_j= I_\epsilon \times T $ to be the product integral current
on
$Z_j =  I_\epsilon\times X$ where
$I_\epsilon=[-\epsilon_j, \epsilon_j]$
(see \cite{Sormani-properties}
for the precise definition of such products of intervals with
currents).   When $T$ is just integration over
a smooth manifold $M$, then $I_\epsilon\times T$
is just integration over $I_\epsilon\times M$.   

In \cite{Sormani-properties}
it is proven that
\be
\partial( I_\epsilon\times T)=I_\epsilon\times(\partial T)
+ (\partial I_\epsilon)\times T.
 \ee
Since $\partial T=0$ we have
\be
\partial B_j=\phi_{j\#}T-\phi'_{j\#}T.
\ee

Then by the definition of the intrinsic flat distance,
\begin{eqnarray}
d_{\mathcal{F}}\left((X, d_j, T), (X, d_\infty, T)\right)
&\le & d_F^{Z_j}(\phi_{j\#}T,\phi'_{j\#}T)\\
&\le & \mass_{(Z_j,d'_j)} (B) + 0.
\end{eqnarray}
So we need only estimate the mass of $B_j$.

In \cite{Sormani-properties} it is shown that
\be
\mass_{(Z_j,D_j)}([-\epsilon_j,\epsilon_j]\times T) = 2\epsilon_j \,\mass_{(X,\lambda d_0)}(T)
\ee
when the distance, $D_j$ is the isometric product metric
on $Z_j$ defined with $d_0$:
\be
D_j((t_1, p_1),(t_2, p_2))=\sqrt{ |t_1-t_2|^2
                                        + (\lambda d_0(p_1, p_2))^2 }.
\ee
Since 
\begin{eqnarray}
d'_j(z_1,z_2) &\le& d'_0((t_1,p_1),(t_2, p_2)):= |t_1-t_2|+ \lambda d_0(p_1, p_2)\\
&\le& \,\,\sqrt{2} \,\,D_j((t_1, p_1),(t_2, p_2)).
\end{eqnarray}
We have
\begin{eqnarray}
\mass_{(Z_j,d'_j)} (B) &\le& \mass_{(Z_j,\sqrt{2}D)} (B) \\
&\le& 2^{(n+1)/2} \,\mass_{(Z_j,D_j)} (B) \\
&\le& 2^{(n+1)/2} 2\epsilon_j \,\mass_{(X,\lambda d_0)}(T)\\
&\le &2^{(n+1)/2} \lambda^{n+1} 2\epsilon_j \,\mass_{(X,d_0)}(T).
\end{eqnarray}
Thus we have (\ref{Fj}) which implies (\ref{Fjlim}).

This completes the proof of Theorem~\ref{app-thm}.
\end{proof}

Finally we prove Lemma~\ref{common-Z}:

\begin{proof}
First note that
\begin{eqnarray}
\qquad d_{-+}          &=& |t_1+\epsilon_j|+|t_2-\epsilon_j|+2\epsilon_j +\inf\left\{d_j(p_1,p)+d_\infty(p,p_2): \,p\in X\right\}\\
\qquad d_{+-} &=& |t_1-\epsilon_j|+|t_2+\epsilon_j|+2\epsilon_j +\inf\left\{d_\infty(p_1,p)+d_j(p,p_2): \,p\in X\right\}.
\end{eqnarray}

Observe that $d'_j$ is immediately symmetric and nonnegative. It is positive definite because 
$$
\min\left\{d(z_1, z_2), d_-(z_1, z_2), d_+(z_1, z_2)\right\} \ge |t_1-t_2|  + \min\left\{d_j(p_1,p_2), d_\infty(p_1,p_2)\right\}
$$
and clearly $d_{-+}(z_1, z_2), d_{+-}(z_1, z_2)  > 2\epsilon_j$ for distinct $z_1, z_2$.

Before proving the triangle inequality, we 
apply (\ref{epsj}) to prove
(\ref{iso-}):
\begin{eqnarray*}
d_-((-\varepsilon_j,p_1), (-\varepsilon_j,p_2))\,\,&=&d_j(p_1, p_2)\\
d\,\,((-\varepsilon_j,p_1), (-\varepsilon_j,p_2))\,\,\,&\ge& d_j(p_1, p_2)\\
d_+((-\varepsilon_j,p_1), (-\varepsilon_j,p_2))\,\,&=&4\epsilon_j+d_\infty(p_1,p_2)\\
&\ge& 4\epsilon_j+ d_j(p_1,p_2)-\epsilon_j\ge d_j(p_1,p_2)\\
d_{-+}((-\varepsilon_j,p_1), (-\varepsilon_j,p_2))&=&
 0+ 2\epsilon_j+2\epsilon_j +\inf\left\{d_j(p_1,p)+d_\infty(p,p_2): \,p\in X\right\}\\
 &\ge& 4\epsilon_j + d_j(p_1,p_2)-\epsilon_j \ge d_j(p_1,p_2)\\
 d_{+-}((-\varepsilon_j,p_1), (-\varepsilon_j,p_2))&=&
 2\epsilon_j+0+2\epsilon_j +\inf\left\{d_\infty(p_1,p)+d_j(p,p_2): \,p\in X\right\}\\
 &\ge& 4\epsilon_j + d_j(p_1,p_2)-\epsilon_j \ge d_j(p_1,p_2).
\end{eqnarray*}
Naturally (\ref{iso+}) follows in a similar way.

It suffices now to prove the triangle inequality. 

In (\ref{tri-a})-(\ref{tri-b}) we prove the triangle
inequality in the case where:
\be\label{tri-01}
d'_j(z_1,z_2)=\min\left\{d(z_1, z_2), d_-(z_1,z_2), d_+(z_1,z_2)\right\}
\ee
and
\be\label{tri-02}
d'_j(z_2,z_3)=\min\left\{d(z_2, z_3), d_-(z_2,z_3), d_+(z_2,z_3)\right\}.
\ee

Observe that
\begin{eqnarray} \label{tri-a}
 \qquad d'_j(z_1,z_3)&\le &d(z_1,z_3)\\
 &=& |t_1-t_3|+ \max\left\{d_j(p_1,p_3), d_\infty(p_1,p_3)\right\}\\
&\le & |t_1-t_2|+|t_2-t_3|\\
&&+\max\left\{d_j(p_1,p_2)+d_j(p_2,p_3), d_\infty(p_1,p_2)+d_\infty(p_2,p_3)\right\}\\
&\le& |t_1-t_2|+\max\left\{d_j(p_1,p_2), d_\infty(p_1,p_2)\right\}\\
&&+|t_2-t_3|+\max\left\{d_j(p_2,p_3), d_\infty(p_2,p_3)\right\}\\
&=&d(z_1, z_2)+d(z_2,z_3)\\
d'_j(z_1,z_3)&\le &d_-(z_1,z_3)\\
&=& |t_1+\epsilon_j|+|t_3+\epsilon_j| + d_j(p_1,p_3)\\
&\le & |t_1-t_2|+|t_2+\epsilon_j|+ |t_3+\epsilon_j| \\
&&+ d_j(p_1,p_2)+d_j(p_2, p_3)\\
&\le & |t_1-t_2|+\max\left\{d_j(p_1,p_2), d_\infty(p_1,p_2)\right\}\\
&& |t_2+\epsilon_j|+ |t_3+\epsilon_j| +d_j(p_2, p_3)\\
&\le & d(z_1,z_2) + d_-(z_2,z_3)
\end{eqnarray}
and similarly
\be
d'_j(z_1,z_3) \le d(z_1,z_2)+d_+(z_2,z_3).
\ee
Clearly
\be
|t_1+\epsilon_j|+|t_3+\epsilon_j|\le 
|t_1+\epsilon_j|+2 |t_2+\epsilon_j|+|t_3+\epsilon_j|
\ee
so
\begin{eqnarray}
d'_j(z_1,z_3) &\le& d_-(z_1,z_2)+d_-(z_2,z_3)\\
d'_j(z_1,z_3) &\le& d_+(z_1,z_2)+d_+(z_2,z_3).
\end{eqnarray}
Immediately by the definition we have
\begin{eqnarray}
&&d'_j(z_1,z_3) \le d_{-+}(z_1, z_3)\le  d_-(z_1,z_2)+d_+(z_2,z_3)\\
&&d'_j(z_1,z_3) \le d_{+-}(z_1, z_3)\le  d_+(z_1,z_2)+d_-(z_2,z_3).
\label{tri-b}
\end{eqnarray}
Thus we have shown the triangle inequality holds
as long as (\ref{tri-01})-(\ref{tri-02}) hold.

We need only prove the triangle inequality for
all the five cases where
\be
d_j'(z_1, z_2)=d_{-+}(z_1, z_2).
\ee
The rest of the cases will follow by symmetry in the
definitions of $d_{-+}$ and $d_{+-}$ and in swapping of
the points $z_1, z_2$ with $z_3, z_2$.

\begin{eqnarray}
\qquad d_j'(z_1, z_3) &\le & d_{-+}(z_1, z_3)\\
&=& |t_1+\epsilon_j|+|t_3-\epsilon_j|+2\epsilon_j +\inf\left\{d_j(p_1,p)+d_\infty(p,p_3): \,p\in X\right\}\\
&\le & |t_1+\epsilon_j|+|t_2-\epsilon_j|+2\epsilon_j +|t_3-t_2|\\
&&+\inf\left\{d_j(p_1,p)+d_\infty(p,p_2) + d_\infty(p_2,p_3): \,p\in X\right\}\\
&\le& d_{-+}(z_1, z_2) + d(z_2, z_3)
\end{eqnarray}

\begin{eqnarray}
d_j'(z_1, z_3) &\le & d_{-+}(z_1, z_3)\\
&= & \inf\left\{d_-(z_1,z)+d_+(z,z_3):\, z\in Z\right\}\\
&\le & \inf\left\{d_-(z_1,z)+d_+(z,z_2)+ d_+(z_2, z_3):\, z\in Z_j\right\}\\
&= & d_{-+}(z_1,z_2) + d_+(z_2, z_3)
\end{eqnarray}

Below we will use the following inequality, which follows from \eqref{epsj},
\begin{eqnarray}
d_j (p_1, p_2) \le \inf\left\{ d_j (p_1, p) + d_{\infty}(p, p_2): p\in X\right\} +  \epsilon_j.
\end{eqnarray}
So
\begin{eqnarray}
d_j'(z_1, z_3) &\le & d_{-}(z_1, z_3)\\
&= & |t_1+\epsilon_j| + |t_3 + \epsilon_j | + d_j (p_1, p_3)\\
&\le & |t_1+\epsilon_j | + |t_3+ \epsilon_j | + d_j(p_1, p_2) + d_j (p_2, p_3)\\
&\le &  |t_1+\epsilon_j | + |t_3+ \epsilon_j | \\
&& +  \inf \left\{ d_j(p_1, p) + d_{\infty} (p, p_2): p\in X\right\} + \epsilon_j+ d_j (p_2, p_3)\\
&\le& d_{-+}(z_1,z_2) + d_-(z_2, z_3)
\end{eqnarray}
and
\begin{eqnarray}
\qquad d_j'(z_1, z_3) &\le & d_{-+}(z_1, z_3)\\
&=& |t_1+\epsilon_j|+|t_3-\epsilon_j|+2\epsilon_j +\inf\left\{ d_j(p_1,p)+d_\infty(p,p_3): p\in X\right\}\\
&\le&  |t_1+\epsilon_j|+|t_3-\epsilon_j|+2\epsilon_j \\
&& +  \inf\left\{ d_j(p_1, p') + d_j(p', p_2): p'\in X\right\}  \\
&&+ \inf\left\{ d_j(p_2, p) + d_{\infty} (p, p_3): p\in X\right\}\\
&\le&  |t_1+\epsilon_j|+|t_3-\epsilon_j|+3\epsilon_j \\
&& +  \inf\left\{ d_j(p_1, p') + d_\infty(p', p_2): p'\in X\right\}  \\
&&+ \inf\left\{ d_j(p_2, p) + d_{\infty} (p, p_3): p\in X\right\}\\
&\le & d_{-+}(z_1,z_2) + d_{-+}(z_2, z_3)
\end{eqnarray}
and
\begin{eqnarray}
d_j'(z_1, z_3) &\le & d_{-}(z_1, z_3)\\
&=& |t_1 + \epsilon_j| + |t_3 + \epsilon_j| + d_j (p_1, p_3) \\
&\le& |t_1 + \epsilon_j| + |t_3 + \epsilon_j| + d_j (p_1, p_2) + d_j (p_2, p_3)\\ 
&\le & |t_1 + \epsilon_j| + |t_3 + \epsilon_j| + 2\epsilon_j\\
&& + \inf\left\{ d_j (p_1, p) + d_\infty (p, p_2): p\in X\right\} \\
&&+ \inf\left\{ d_j (p_2, p) + d_\infty (p, p_3): p\in X\right\} \\
&\le& d_{-+}(z_1,z_2) + d_{+-}(z_2, z_3).
\end{eqnarray}
Thus $d_j'$ is a metric.
\end{proof}

The metric space $Z_j$ constructed in Lemma~\ref{common-Z}
is not necessarily complete even if $X$ is complete with respect
to both $d_j$ and $d_\infty$:  

\begin{example} \label{not-complete}
Let $X=\{0,1/2,1/4,...\} \cup \{1\}$.
Let $d_j(p_1, p_2)=|p_1-p_2|$.  Let $F:X \to X$ be the identity
map on $X\setminus \{0,1\}$ and $F(0)=1$ and $F(1)=0$.  Let
$d_\infty(p_1,p_2)=|F(p_1)-F(p_2)|$.   Both $(X, d_j)$ and
$(X, d_\infty)$ are complete but with different limits for the
sequence $\{1/2, 1/4,...\}$:
\[
	d_j (1/i, 0) \to 0 \quad \mbox{and} \quad d_{\infty} (1/i, 1) \to 0\quad \quad \mbox{as } i \to \infty.
\]

Observe that $\epsilon_j= 1$ because 
\be
1\ge \epsilon_j \ge \lim_{i\to \infty} |d_j(1/i,1)-d_\infty(1/i,1)|=1.
\ee
So 
\be
Z_j= [-1,1] \times X.
\ee
Take the sequence of points $z_i=(0, 1/i)$.   This sequence
is Cauchy in $Z_j$ because
\be
d'_j(z_i,z_k) \le d(z_i,z_k) = 0+ |1/i-1/k| \quad \forall i,k >1.
\ee

Assume on the contrary that this sequence of points converges
to a point $z_\infty=(t_\infty, p_\infty)\in Z_j$.   Observe that for any $z\in Z_j$, 
 \begin{eqnarray}
d_+(z_i,z) &\ge& |0-1|+|t-1|\ge 1\\
d_-(z_i,z) &\ge & |0+1|+|t+1|\ge 1\\
d_{-+}(z_i,z_\infty)&=&\inf\left\{d_-(z_i,z)+d_+(z,z_\infty):\, z\in Z_j\right\} \ge 1\\
d_{+-}(z_i,z_\infty)&=&\inf\left\{d_+(z_i,z)+d_-(z,z_\infty):\, z\in Z_j\right\} \ge 1.
\end{eqnarray}
Therefore, for $i$ sufficiently large, 
\be
d'_j(z_i,z_\infty) = d(z_i,z_\infty)=
|0-t_\infty|+ \max\left\{d_j(1/i, p_\infty), d_\infty(1/i,p_\infty)\right\} \to 0.
\ee
Thus $p_{\infty}$ is the limit of the sequence $\{1/i\}$ with respect to both metrics $d_j, d_{\infty}$, which is a contradiction.
Thus $Z_j$ is not complete.

The metric completion of $Z_j$ is
\be
\bar{Z}_j=[-1,1] \times (X\cup \{p_\infty\})  \, |_\sim
\ee
where $(-1,p_\infty)\sim(-1,0)$ and $(1, p_\infty)\sim(1,1)$.
For $t_i\in [-1,1]$ and $p_i \in X$ we have
\begin{eqnarray}
d_j'((t_1,p_1), (t_2,p_2)) 
&=& \textrm{ as in Lemma~\ref{common-Z}}\\
d_j'((t_1,p_1), (t_2,p_\infty)) 
&=& \lim_{k\to \infty} d_j'((t_1,p_1), (t_2,1/k))\\
d_j'((t_1,p_\infty), (t_2,p_\infty))   
&=& \lim_{k\to \infty} d_j'((t_1,1/k), (t_2,1/k))
\end{eqnarray}
Note that with this distance
\begin{eqnarray}
d_j'((-1,0), (-1,p_\infty)) 
&=& \lim_{k\to \infty} d_j'((-1,0), (-1,1/k))\\
&=& \lim_{k\to \infty} d_j(0, 1/k)=d_j(0, 0)=0 \\
d_j'((1,1),(1, p_\infty))
&=& \lim_{k\to \infty} d_j'((1,1), (1,1/k))\\
&=& \lim_{k\to \infty} d_\infty(1, 1/k)=d_j(0, 1/k)=0 \\
\end{eqnarray}
and that is why $(-1,p_\infty)\sim(-1,0)$ and $(1, p_\infty)\sim(1,1)$.
\end{example}

\bibliographystyle{alpha}
\bibliography{2014}

\end{document}